\documentclass{article}

\usepackage{intro}
\usepackage[margin=1in]{geometry}
\usepackage{cite}
\usepackage{hyperref}
\usepackage{url}
\usepackage{tikz-cd}

\title{Breuil--M\'ezard conjectures for central division algebras}
\author{Andrea Dotto}
\date{}

\begin{document}

\maketitle

\begin{abstract}
We formulate an analogue of the Breuil--M\'ezard conjecture for the group of units of a central division algebra over a $p$-adic local field, and we prove that it follows from the conjecture for~$\GL_n$. To do so we construct a transfer of inertial types and Serre weights between the maximal compact subgroups of these two groups, in terms of Deligne--Lusztig theory, and we prove its compatibility with mod~$p$ reduction, via the inertial Jacquet--Langlands correspondence and certain explicit character formulas. We also prove analogous statements for $\ell$-adic coefficients.
\end{abstract}

\tableofcontents

\section{Introduction.}
Let~$F/\bQ_p$ be a finite extension. The Breuil--M\'ezard conjecture, as originally formulated in~\cite{BMmultiplicites} and generalized in~\cite{Kisinstructure} and~\cite{EGBM}, provides a description of the singularities of potentially semistable deformation rings for~$G_F = \Gal(\overline{F}/F)$ in terms of the representation theory of maximal compact subgroups of~$\GL_n(F)$. In~\cite{GGBM} the authors raise the question of whether an analogous statement holds for the unit groups of central division algebras, and answer it affirmatively for quaternion algebras, proving that it would follow from the truth of the conjecture for~$\GL_2(F)$. This acquires particular relevance in light of the work of Scholze~\cite{ScholzeLT} and Chojecki--Knight~\cite{CKJL} on $p$-adic Jacquet--Langlands correspondences.

In this paper we prove similar results for an arbitrary central division $F$-algebra~$D$. 
Recall that the Jacquet--Langlands correspondence is a bijection from the irreducible smooth representations of~$D^\times$ to the essentially square-integrable representations of~$\GL_n(F)$, characterized by an equality of characters on matching regular elliptic elements. It is compatible with unramified twists, hence it induces a map on inertial equivalence classes. 
Under the local Langlands correspondence for~$\GL_n(F)$, the inertial classes correspond to \emph{inertial types}, which are smooth representations of the inertia group extending to the Weil group.
The correspondence is such that two Weil--Deligne representations are Langlands parameters of inertially equivalent representations if and only if their underlying $W_F$-representations have isomorphic restrictions to inertia. 
If~$\rhobar: G_F \to \GL_n(\cbF_p)$ is a continuous representation, the choice of an $n$-dimensional inertial type~$\tau$ and a dominant cocharacter~$\lambda$ for~$\Res_{F/\bQ_p}\GL_{n, F}$ defines a quotient of the universal lifting ring~$R^\square_{\rhobar}$ of~$\rhobar$, whose points in characteristic zero correspond to potentially semistable lifts of~$\rhobar$ with Hodge type~$\lambda$ and inertial type~$\tau$. The Breuil--M\'ezard conjecture is concerned with the cycles that the mod~$p$ fibres of these rings define on $\Spec R^\square_{\rhobar}$.

To be more precise, recall that work of Henniart (appendix to~\cite{BMmultiplicites}) and Schneider--Zink \cite{SZtypes} associates to~$\tau$ certain smooth representations~$\sigma_{\fP}(\tau)$ of~$\GL_n(\mO_F)$, which refine the Bushnell--Kutzko theory of types by taking into account the monodromy operator on Langlands parameters. 
On the side of~$\GL_n(F)$, these types compute the shape of a partition~$\fP(\pi)$ attached to a representation~$\pi$ by Bernstein and Zelevinsky. 
We will only be concerned with the case of~$\tau$ corresponding to an inertial class of the form~$\fs(\tau) = \left [ \prod_{i=1}^r \GL_{n/r}(F), \pi_0^{\otimes  r} \right ]$ where~$\pi_0$ is a supercuspidal representation (these are precisely the inertial classes containing discrete series representations). In this case, $\sigma_{\fP}(\tau)$ has the property that, for a generic representation~$\pi$ of~$\GL_n(F)$, the space $\Hom_{\GL_n(\mO_F)}(\sigma_{\fP}(\tau), \pi)$ is not zero if and only if~$\pi$ is supported in~$\fs(\tau)$ and~$\fP(\pi) \geq \fP$ in the reverse of the dominance partial order on partitions of~$r$. The maximal partition~$\fP_{\max}$ is $r = 1+\cdots +1$, and for a generic~$\pi$ the partition~$\fP(\pi)$ is maximal if and only if the monodromy operator on the Langlands parameter~$\rec(\pi)$ equals zero. 

In line with this, \cite{EGBM} asks for the existence of a map
\begin{displaymath}
R_{\cbF_p}(\GL_n(\mO_F)) \to Z(R^\square_{\rhobar}/\pi)
\end{displaymath}
from the Grothendieck group of finite length $\cbF_p$-representations of~$\GL_n(\mO_F)$ to the group of cycles on~$R^\square_{\rhobar}/\pi$, such that the image of the semisimplified mod~$p$ reduction $\sigmabar_{\fP_{\max}}(\tau)$ of~$\sigma_{\fP_{\max}}(\tau)$ is $Z(R^\square_{\rhobar}(\tau, 0)_{\mathrm{cris}}/\pi)$, the cycle attached to the mod~$p$ fibre of the potentially \emph{crystalline} deformation ring with inertial type~$\tau$ and $\lambda = 0$ (one should work with coefficients in a large finite extension $E/\bQ_p$, and we do so in the paper, so that~$\pi$ is a uniformizer of~$E$). There is a similar statement for arbitrary~$\lambda$, by tensoring~$\sigma_{\fP_{\max}}(\tau)$ with the corresponding algebraic representation.

One might guess that the extension of this to semistable representations will relate~$\sigmabar_{\fP}(\tau)$ to the mod~$p$ fibre of the strata~$R_{\rhobar}^\square(\tau, \lambda)_{\fP}$ induced by the monodromy operator on the universal $\varphi, N$-module on~$R^\square_{\rhobar}(\tau, \lambda)$, which are again classified by partitions. 
However, we have found that one needs to be slightly careful in formulating this, and work instead with a virtual representation~$\sigma_{\fP}^+(\tau)$ closely related to the Schneider--Zink types. 
It has the property that, for a generic representation~$\pi$ of~$\GL_n(F)$,
\begin{align*}
\dim \Hom_{\bK}(\sigma^+_\fP(\tau), \pi) &= 1 \; \text{if~$\pi$ has inertial class~$\fs(\tau)$ and $\fP(\pi) = \fP$}\\
&= 0 \; \text{otherwise}.
\end{align*}
That these representations appear is consistent with the work of Shotton~\cite{ShottonBM} in the case of $\ell$-adic coefficients for~$\ell \not = p$. 

\paragraph{Main results.} 
With the above discussion in place, we can state our main results. 
The characteristic zero points of the stratum~$R^\square_{\rhobar}(\tau, \lambda)_{\fP_{\min}}$ indexed by the minimal partition~$\fP_{\min}$ of~$r$ correspond to potentially semistable lifts of the representation~$\rhobar$ whose Weil--Deligne representation is the Langlands parameter of an essentially square-integrable representation, and these can be transferred to~$D^\times$. 
Indeed, let~$\tau$ be a discrete series inertial type of dimension~$n$, corresponding to an inertial class~$\fs(\tau)$ of $\GL_n(F)$-representations. 
The Jacquet--Langlands correspondence provides an inertial class $\fs_D(\tau) = \JL^{-1}\fs(\tau)$ of representations of~$D^\times$, which admits types on the maximal compact subgroup~$\mO_D^\times$. 
In contrast with the case of~$\GL_n(F)$, they are not uniquely determined, and we write~$\sigma_D(\tau)$ for an arbitrarily chosen one: our results apply to all possible choices of~$\sigma_D(\tau)$. 
The weight~$\lambda$ also determines a representation of~$\mO_D^\times$, and we write~$\sigma_D(\tau, \lambda)$ for the tensor product of the two.

\begin{thmintro}[Breuil--M\'ezard conjecture for~$D^\times$, see Section~\ref{BMconjectures}]
If the geometric Breuil--M\'ezard conjecture holds for~$\GL_n(F)$, then there exists a group homomorphism
\begin{displaymath}
R_{\cbF_p}(\mO_D^\times) \to Z(R^\square_{\rhobar}/\pi)
\end{displaymath}
which for any~$(\tau, \lambda)$ sends the semisimplified mod~$p$ reduction $\sigmabar_D(\tau, \lambda)$ of~$\sigma_D(\tau, \lambda)$ to $Z(R^\square_{\rhobar}(\tau, \lambda)_{\fP_{\min}}/\pi)$.
\end{thmintro}

The theorem is proved following the same strategy as~\cite{GGBM}, but the techniques we use are different. We begin by constructing a group homomorphism
\begin{displaymath}
\JL_p: R_{\cbF_p}(\mO_D^\times) \to R_{\cbF_p}(\GL_n(\mO_F))
\end{displaymath}
via Deligne--Lusztig induction, and describing it in terms of the combinatorics of parabolic induction. 
Our main technical result is Theorem~\ref{comparison}, stating the equality $\JL_p(\sigmabar_D(\tau, \lambda)) = \sigmabar^+_{\fP_{\min}}(\tau, \lambda)$. 
Granting this, one transfers the result from~$\GL_n(F)$ to~$D^\times$ by composing with~$\JL_p$. 
In order to prove Theorem~\ref{comparison} we need a complete description of the Jacquet--Langlands correspondence in terms of type theory, which was obtained in~\cite{InertialJL}.
We deduce our result from this description, a base change procedure to unramified extensions of~$F$ originating in~\cite{BHliftingI}, and explicit computations with a number of character formulas. 

\paragraph{A Jacquet--Langlands transfer on maximal compact subgroups.}
Since~$F^\times \mO_D^\times$ is a normal subgroup of~$D^\times$ with finite cyclic quotient, one proves that every smooth irreducible representation of~$\mO_D^\times$ with complex coefficients is a type for a Bernstein component of~$D^\times$. 
It follows that our constructions in type theory give rise to a group homomorphism
\begin{displaymath}
\JL_{\bK}: R_{\cbQ_p}(\mO_D^\times) \to R_{\cbQ_p}(\GL_n(\mO_F))
\end{displaymath}
and our main results imply that the following diagram commutes. See Section~\ref{JLtransfer} for details.
\begin{equation}\label{diagramp}
\begin{tikzcd}
R_{\cbQ_p}(\mO_D^\times) \arrow[r, "\JL_{\bK}"] \arrow[d, "\br_p"] & \arrow[d, "\br_p"] R_{\cbQ_p}(\GL_n(\mO_F)) \\
R_{\cbF_p}(\mO_D^\times) \arrow[r, "\JL_p"] & R_{\cbF_p}(\GL_n(\mO_F))
\end{tikzcd}
\end{equation}

After a first version of this paper was written, we have been notified of work in preparation of Zijian Yao that makes the following equivalent construction. 
Consider the abelian group $\bigoplus_{(\tau, N)}\bZ$ where the sum is indexed by Galois inertial types~$\tau$ with monodromy operator~$N$. There is a map $R_{\cbQ_p}(\GL_n(\mO_F)) \to \bigoplus_{(\tau, N)}\bZ$, sending a representation~$\sigma$ to 
\begin{displaymath}
\left ( \dim_{\cbQ_p}\Hom_{\GL_n(\mO_F)}(\sigma, \pi_{\tau, N}) \right )_{(\tau, N)}
\end{displaymath}
for any generic irreducible representation $\pi_{\tau, N}$ such that $\rec(\pi_{\tau, N})$ has inertial type~$\tau$ and monodromy operator~$N$. 
By definition, our representations $\sigma_{\fP}^+(\tau)$ yield a section of this map. 
There is an analogous map defined for~$\mO_D^\times$, whose image is contained in the direct sum of the factors indexed by discrete series inertial types. 
Yao defines~$\JL_{\bK}$ as the map making the following diagram commute
\begin{equation}
\begin{tikzcd}
R_{\cbQ_p}(\mO_D^\times) \arrow[dr] \arrow[rr, "\JL_{\bK}"] & & R_{\cbQ_p}(\GL_n(\mO_F)) \\
& \bigoplus_{(\tau, N)}\bZ \arrow[ur, "\sigma_{\fP}^+(\tau)"] &
\end{tikzcd}
\end{equation}
and goes on to conjecture the existence of a map~$\JL_p$ making diagram~(\ref{diagramp}) commute. Our results therefore provide a proof of this.

Yao makes similar conjectures in the case of more general inner forms, as this definition of~$\JL_{\bK}$ makes sense for~$\GL_r(D)$ when formulated for those inertial types~$(\tau, N)$ extending to a Langlands parameter for~$\GL_r(D)$. At least in the case of discrete series parameters, it seems that our methods extend to this situation without too much trouble: the inertial Jacquet--Langlands correspondence is proved in full generality in \cite{InertialJL}, and there is a natural candidate for the~$\JL_p$ map, namely Lusztig induction for the twisted Levi subgroup~$\GL_r(\bd)$ of~$\GL_n(\mbf)$. 
We have chosen to focus on the simpler case of~$D^\times$: amongst other reasons for this choice, we remark that from the viewpoint of a Jacquet--Langlands correspondence for maximal compact subgroups one expects weaker results for~$\GL_r(D)$ than for~$D^\times$. 
For instance, not every irreducible representation of~$\GL_r(\mO_D)$ is a type for~$\GL_r(D)$, and~$\JL_{\bK}$ does not see any information about non-typical representations of~$\GL_r(\mO_D)$, except their multiplicities in restrictions of $\GL_r(D)$-representations.

\paragraph{}
We have the following parallel statement for $\ell$-adic coefficients when~$\ell \not = p$.

\begin{thmintro}(Theorem~\ref{modlJL}.)
There exists a (necessarily unique) morphism~$\JL_\ell$ making the following diagram commute.
\begin{equation}
\begin{tikzcd}
R_{\cbQ_\ell}(\mO_D^\times) \arrow[r, "\JL_{\bK}"] \arrow[d, "\br_\ell"] & \arrow[d, "\br_\ell"] R_{\cbQ_\ell}(\GL_n(\mO_F)) \\
R_{\cbF_\ell}(\mO_D^\times) \arrow[r, "\JL_\ell"] & R_{\cbF_\ell}(\GL_n(\mO_F))
\end{tikzcd}
\end{equation}
\end{thmintro}

The uniqueness statement follows from the fact that the reduction mod~$\ell$ map for~$\mO_D^\times$ is surjective. In fact we can give an explicit description of~$\JL_\ell$ in terms of~$\JL_{\bK}$, and (as usual) the theorem has no new content when~$\ell$ does not divide the pro-order of~$\GL_n(\mO_F)$, since in this case both vertical arrows are isomorphisms. It is worth stating explicitly a difference with the case $\ell = p$. The mod~$p$ irreducible representations of~$\mO_D^\times$ are characters, and they lift to level zero types for~$D^\times$. Hence compatibility with the Jacquet--Langlands transfer of level zero types already determines the~$\JL_p$ map uniquely, and compatibility for all types imposes a strong constraint on their mod~$p$ reductions. When~$\ell \not = p$, there are a lot more irreducible $\cbF_\ell$-representations of~$\mO_D^\times$, and the only congruences arise between types with the same endo-class. This allows us to construct $\JL_\ell$ by fixing the endo-class and studying the mod~$\ell$ reduction of the level zero part, which is what is done in the proof of Theorem~\ref{modlJL}.

From our theorem together with~\cite[Theorem~4.6]{ShottonBM} (which requires the assumption that $p \not = 2$), we deduce that a form of the geometric Breuil--M\'ezard conjecture holds for~$D^\times$ and $\ell$-adic coefficients, expressing the fact that congruences between the special fibres of discrete series deformation rings are described by mod~$\ell$ congruences between types on the maximal compact subgroup of~$D^\times$. See Theorem~\ref{modlBM}.

\paragraph{Structure of the article.} The paper is organized as follows. 
Section~\ref{finitegroups} is about Deligne--Lusztig theory and begins with definitions and some simple results that are certainly well-known but we could not find in the literature in the exact form we needed (although for instance~\cite[1.18]{LusztigCoxeter} is closely related). Then we specialize to~$\GL(n)$: we study the structure of parabolic induction, give a character formula~(Proposition~\ref{DLinduction}) and construct the representations~$\sigma_{\fP}^+$ (Theorem~\ref{sigmamax}). 
Section~\ref{typetheory} recalls the results of~\cite{SZtypes} and proves analogues for~$D^\times$.
We repeat the Schneider--Zink construction for~$\sigma_{\fP}^+$ and construct our virtual representations~$\sigma_{\fP}^+(\tau)$. 
We end with two formulas for the trace of~$\sigma_{\fP_{\min}}^+(\tau)$ and~$\sigma_D(\tau)$ on pro-$p$-regular conjugacy classes of~$\GL_n(\mO_F)$ and~$\mO_D^\times$, and relate them via a computation of formal degrees. 
Section~\ref{Galoisdeformationtheory} recalls the monodromy stratification (see also~\cite{Pyvovarovthesis}) and states the geometric Breuil--M\'ezard conjecture for potentially semistable deformation rings. 
The connection with~\cite{ShottonBM} is made explicit. 
Finally, we define our Jacquet--Langlands transfers of weights and types in Section~\ref{JLtransfer} and prove our main theorems in Sections~\ref{JLtransfer} and~\ref{BMconjectures}.

\paragraph{Acknowledgments.} I am grateful to Toby Gee for suggesting the problem and for helpful conversations. I have also benefited from discussions on these and related matters with Emma Knight, Daniel Le, Stefano Morra and Zijian Yao at the thematic trimester on Algebraic Groups and Geometrization of the Langlands Program at ENS Lyon. Finally, thanks are due to the referee for several useful comments. This work was supported by the Engineering and Physical Sciences Research Council [EP/L015234/1], The EPSRC Centre for Doctoral Training in Geometry and Number Theory (The London School of Geometry and Number Theory), University College London, and Imperial College London; as well as by a Royal Society University Research Fellowship.

\paragraph{Notation and conventions.}
We use the same notation as~\cite{InertialJL}, so that if~$F$ is a local field we write~$\mbf$ for its residue field and~$\mu_F$ for the group of Teichm\"uller (i.e. prime-to-$p$) roots of unity in~$F^\times$. We fix an algebraic closure~$\overline{F}$ of~$F$, and write~$F_n$ for the unramified extension of~$F$ in~$\overline{F}$ of degree~$n$ and~$\mbf_n$ for its residue field. In general, $\bk_n$ denotes an extension of the finite field~$\bk$ of degree~$n$. A character of~$\bk_n^\times$ is called $\bk$-regular if its orbit under the action of~$\Gal(\bk_n/\bk)$ has~$n$ distinct elements.
For an endo-class~$\Theta_F$ over~$F$ we write~$\delta(\Theta_F)$ for the degree over~$F$ of a parameter field of~$\Theta_F$, $e(\Theta_F)$ for its ramification index and~$f(\Theta_F)$ for its residue field degree. We write~$\bK$ for the maximal compact subgroup~$\GL_n(\mO_F)$ of~$\GL_n(F)$.

We consider partitions of a positive integer~$n$ as functions $\fP: \bZ_{>0} \to \bZ_{\geq 0}$ with finite support, such that $\sum_{i \in \bZ_{>0}}i\fP(i) = n$. 
Whenever~$n$ is an integer and~$p$ a prime number, we write~$n_p$ for the highest power of~$p$ dividing~$n$ and~$n_{p'} = n/n_p$.

Parabolic induction from a block-diagonal Levi subgroup of $\GL(n)$ is always taken with respect to the corresponding upper-triangular parabolic subgroup. We consider normalized induction for~$\GL_n(F)$ unless stated otherwise. From Section~\ref{typetheory}, whenever dealing with a finite general linear group~$\GL_n(\bF_q)$ we will write~$R_w$ for the Deligne--Lusztig induction from an elliptic maximal torus (the type of such a torus consists of the $n$-cycles, and its group of rational points is isomorphic to~$\bF_{q^n}^\times$).

Unless stated otherwise, representations will have complex coefficients and representations of locally profinite groups will be smooth. The local Langlands correspondence for~$\GL_n(F)$ is denoted~$\rec$. If~$p$ is a prime number, any choice of an isomorphism $\iota_p : \cbQ_p \to \bC$ gives rise to a local Langlands correspondence~$\rec_{\cbQ_p}$ for smooth representations with $\cbQ_p$-coefficients. This depends on the choice of~$\iota_p$ only up to an unramified twist, hence its behaviour on inertial classes of representations is independent of the choice of~$\iota_p$.

\section{Representation theory of~$\GL_n(\bF_q)$.}\label{finitegroups} Fix a prime number~$p$ and let~$q$ be a power of~$p$. In this section we recall the combinatorial classification, in terms of partitions, of the complex irreducible representations of~$G = \GL_n(\bF_q)$ with simple supercuspidal support, following~\cite[Sections~3,4]{SZtypes}. We give a construction, in terms of Deligne--Lusztig theory, of a certain virtual representation with special properties with respect to this classification, which will appear in the construction of the element of $R_{\cbF_p}(\GL_n(\mO_F))$ corresponding to the minimal stratum of a Galois deformation ring.

\paragraph{Harish--Chandra series.} Every irreducible representation~$\pi$ of~$G$ has a supercuspidal support, which is unique up to conjugacy. The \emph{simple} supercuspidal supports are those conjugate to
\begin{displaymath}
r\pi_0 = \left ( \prod_{i=1}^r \GL_{n/r}(\bF_q), \pi_0^{\otimes r} \right)
\end{displaymath} 
for some positive divisor~$r$ of~$n$ and some supercuspidal representation~$\pi_0$ of~$\GL_{n/r}(\bF_q)$. There exists a unique nondegenerate representation supported in~$r\pi_0$, denoted $\St(\pi_0, r)$. To classify the others, we consider partitions~$\fP$ of~$r$, and to each~$\fP$ we associate a block-diagonal Levi subgroup
\begin{displaymath}
L_{\fP}(\pi_0) = \prod_{i \in \bZ_{>0}}\GL_{ni/r}(\bF_q)^{\times \fP(i)}
\end{displaymath}
and a parabolically induced representation of~$\GL_n(\bF_q)$
\begin{displaymath}
\pi_{\fP}(\pi_0) = \times_{i \in \bZ_{>0}} \St(\pi_0, i)^{\times \fP(i)}. 
\end{displaymath}
The partition~$\fP_{\max}$ sending~$1$ to~$r$ and every other positive integer to~$0$ corresponds to writing~$r$ as a sum of~$1$. The representation~$\pi_{\fP_{\max}}(\pi_0)$ is the full parabolic induction $\pi_0^{\times r}$. The Harish-Chandra series corresponding to~$r\pi_0$ is the set of irreducible representations of~$G$ with supercuspidal support~$r\pi_0$.
It coincides with the set of Jordan--H\"older factors of~$\pi_{\fP_{\max}}(\pi_0)$.

Write $\fP' \leq \fP$ for the reverse of the dominance partial order on partitions, as in~\cite{SZtypes}. Then~$\fP_{\max}$ is the maximal element amongst partitions of~$r$. There is a bijection $\fP \mapsto \sigma_{\fP}(\pi_0)$ from the set of partitions of~$r$ to the Harish-Chandra series for~$r\pi_0$, characterized by the fact that~$\sigma_{\fP}(\pi_0)$ occurs in~$\pi_{\fP'}(\pi_0)$ if and only if~$\fP \leq \fP'$, and it occurs in~$\pi_{\fP}(\pi_0)$ with multiplicity one. The smallest element amongst partitions of~$r$ is denoted~$\fP_{\min}$ and sends~$r$ to~$1$ and every other positive integer to~$0$. We have $\sigma_{\fP_{\min}}(\pi_0) = \pi_{\fP_{\min}}(\pi_0) = \St(\pi_0, r)$.

When $\fP \leq \fP'$, the multiplicity of $\sigma_{\fP}(\pi_0)$ in $\pi_{\fP'}(\pi_0)$ is by definition the \emph{Kostka number~$K_{\fP, \fP'}$}. It depends only on the two partitions~$\fP, \fP'$, and not on the representation~$\pi_0$.
More precisely, the standard definition of the Kostka numbers is formulated in terms of the representation theory of symmetric groups, as in~\cite[Section~6.1]{ShottonBM}, and it is related to the representation theory of finite general linear groups in~\cite[Corollary~6.10]{ShottonBM}.
Our normalizations coincide with~\cite[Definition~6.2]{ShottonBM}, since the partial order that appears there is the reverse of~$\leq$.

\paragraph{Lusztig induction.}

We follow the presentation of Deligne--Lusztig theory in~\cite{DMbook}. 
The material in this paragraph is mostly standard, but we need to fix notations and to provide certain results about products and Weil restriction of scalars that are probably well-known but we could not find in the literature (although~\cite[1.18]{LusztigCoxeter} is closely related).
So we have decided to provide the proofs.

Let~$\bG_0$ be a connected reductive group over~$\bk = \bF_q$, fix an algebraic closure~$\cbk$ of~$\bk$, and write $\bG= \bG_0 \times_{\bk} \cbk$. 
The rational structure~$\bG_0$ gives rise to a $\cbk$-linear Frobenius endomorphism~$F$ of~$\bG$, the pullback of the absolute $q$-th power Frobenius morphism of~$\bG_0$. The Galois group~$\Gal(\cbk / \bk)$ acts to the right on~$\bG$, via $\bF_q$-linear automorphisms. We write~$\varphi$ for the geometric Frobenius element of the Galois group, acting as $x \mapsto x^{1/q}$. If~$\bH$ is a subgroup of~$\bG$ we will write $F\bH$ for the parabolic subgroup $\varphi(\bH)$ of~$\bG$, whose group of $\cbk$-points is~$F(\bH(\cbk))$. We will say that~$\bH$ is \emph{$F$-stable}, or \emph{rational}, if~$F\bH = \bH$. Recall from~\cite[Definition~8.3]{DMbook} the invariant~$\epsilon_{\bG_0} = (-1)^{\eta(\bG_0)}$, where the $\bF_q$-rank~$\eta(\bG_0)$ is the dimension of the maximal split subtorus of any quasisplit rational maximal torus in~$\bG_0$ (the quasisplit rational maximal tori are those contained in a rational Borel subgroup). 

Fix a parabolic subgroup~$\bP$ of~$\bG$, with unipotent radical~$\bU$ and $F$-stable Levi factor~$\bL$ (without assuming that~$\bP$ is $F$-stable). The associated Deligne--Lusztig varieties can be defined in terms of the Lang isogeny
\begin{displaymath}
\mL: \bG \to \bG, x \mapsto x^{-1}F(x)
\end{displaymath}
by setting
\begin{align*}
\bX_{\bL \subset \bP}^\bG &= \mL^{-1}(F\bP) / (\bP \cap F\bP)\\
\bY_{\bL \subset \bP}^\bG &= \mL^{-1}(F \bU) / (\bU \cap F\bU).
\end{align*}
Both varieties have an action of~$\bG^F \cong \bG(\bk)$ by left multiplication, and~$\bY_{\bL \subset \bP}^\bG$ has an action of~$\bL^F \cong \bL(\bk)$ by right multiplication. We write~$H_c^*(\bY_{\bL \subset \bP}^\bG)$ for the alternating sum $\sum_{i \in \bZ} (-1)^i[H^i_c(\bY_{\bL \subset \bP}^\bG, \cbQ_\ell)]$ of compactly supported $\ell$-adic cohomology groups, for a prime number $\ell \not = p$. Each cohomology group carries a left action of~$\bG^F$ and a right action of~$\bL^F$. The associated Lusztig induction functor is
\begin{displaymath}
R_{\bL \subset \bP}^{\bG}: R_{\cbQ_\ell}(\bL^F) \to R_{\cbQ_\ell}(\bG^F), \; [V] \mapsto H_c^*(\bY_{\bL \subset \bP}^\bG) \otimes_{\cbQ_\ell[\bL^F]} V.
\end{displaymath}
On characters, we have the formula (see~\cite[Proposition~4.5]{DMbook})
\begin{displaymath}
R_{\bL \subset \bP}^{\bG}(\theta)(g) = |\bL^F|^{-1} \sum_{l \in \bL^F}\sum_{i \in \bZ}(-1)^i\tr \left ( (g, l) | H^i_c(\bY_{\bL \subset \bP}^\bG, \cbQ_\ell) \right ) \theta(l^{-1}).
\end{displaymath}

\begin{rk}
Since $\bU \cap F\bU$ is an affine space we  obtain the same induction functor via the bimodule $H^*_c(\mL^{-1}(F\bU))$. This is the functor denoted $R_{\bL \subset F\bP}^\bG$ in~\cite{DMbook}, since their~$R_{\bL \subset \bP}^\bG$ is constructed via $H^*_c(\mL^{-1}(\bU))$.
\end{rk}

When~$\bL$ is a maximal torus, there is another description of Lusztig induction via the Bruhat decomposition of~$\bG$. Fix a pair $(\bB, \bT)$ consisting of an $F$-stable maximal torus and an $F$-stable Borel subgroup of~$\bG$ containing~$\bT$. By~\cite[Lemma~1.13]{DLreps} there is a bijection between the $\bG^F$-conjugacy classes of pairs $(\bB', \bT')$ consisting of a Borel subgroup of~$\bG$ and a rational maximal torus of~$\bB'$, and the Weyl group~$W(\bT)$, given by the map $(g \bB g^{-1}, g \bT g^{-1}) \mapsto g^{-1}F(g)$ (here $g \in \bG(\cbk)$). The $F$-conjugacy classes in~$W(\bT)$ are the equivalence classes for $x\sim g x F(g^{-1})$, and they classify $\bG^F$-conjugacy classes of $F$-stable maximal tori in~$\bG$ by~\cite[Corollary~1.14]{DLreps}. For~$w$ in~$W(\bT)$, we write~$\bT_w$ for an $F$-stable maximal torus in~$\bG$ classified by the $F$-conjugacy class of~$w$, and we say that~$w$ is the \emph{type} of~$\bT_w$.

The Bruhat decomposition for~$\bG$ is $\bG = \bigsqcup_{w \in W(\bT)} \bB \dot{w} \bB$ for any choice of representatives~$\dot{w}$ of~$W(\bT)$ in~$\bG$ (it is independent of the choice of~$\dot{w}$). The quotient $\bB w \bB/\bB$ is a Schubert cell in the flag variety~$\bG/\bB$, and there is an associated Deligne--Lusztig variety
\begin{displaymath}
\bX(w) = ( \mL^{-1}(\bB w \bB) )/\bB
\end{displaymath}
together with a covering
\begin{displaymath}
\bY(\dot{w}) = (\mL^{-1}(\bU\dot{w}\bU))/\bU
\end{displaymath}
induced by the canonical surjection $\bG/\bU \to \bG/\bB$. Both varieties have a left multiplication action by~$\bG^F$. If we equip~$\bT$ with the twisted Frobenius endomorphism $wF: t \mapsto w F(x) w^{-1}$, then the group of fixed points $\bT^{wF}$ acts by right multiplication on~$\bY(\dot{w})$. One checks as in \cite[1.8]{DLreps} that the isomorphism class of this covering, together with the action of~$\bT^{wF}$ and~$\bG^F$, is independent of the choice of~$\dot{w}$.

Now consider a pair $(\bB', \bT')$ consisting of a Borel subgroup of~$\bG$ and a rational maximal torus of~$\bB'$, classified as in the above by some~$w \in W(\bT)$. By~\cite[Proposition~1.19]{DLreps} whenever we have~$x \in \bG$ with $(\bB', \bT') = x(\bB, \bT)x^{-1}$ and~$\mL(x) = \dot{w}$, the map $g\mapsto gx^{-1}$ induces an isomorphism $\bY(\dot{w}) \to \bY_{\bT' \subset \bB}^{\bG}$ that is equivariant for the isomorphism $\ad(x) : \bT^{wF} \to (\bT')^F$, and $\bG^F$-equivariant.

It follows that we can attach to each element $w \in W(\bT)$ an induction map
\begin{displaymath}
R_w: R_{\cbQ_\ell}(\bT^{wF}) \to R_{\cbQ_\ell}(\bG^F)
\end{displaymath}
via the cohomology $H^*_c(\bY(\dot{w}))$ for any representative~$\dot{w}$ of~$w$.

\paragraph{} We need to study the behaviour of the maps~$R_w$ with respect to Weil restriction of scalars and products. 
Define $\bG_n = \bG_0 \times_{\bk} \bk_n$ and
\begin{displaymath}
\bG^+_0 = \Res_{\bk_n/\bk} \left ( \bG_0 \times_{\bk} \bk_n \right ).
\end{displaymath}
The base change $\bG^+ = \bG^+_0 \times_{\bk} \cbk$ is isomorphic to a product $\prod_{i=1}^n \bG$, and its Frobenius endomorphism acts (on $R$-points, for any $\cbk$-algebra~$R$) by
\begin{displaymath}
(g_1, \ldots, g_m) \mapsto (F(g_m), F(g_1), \ldots, F(g_{m-1}))
\end{displaymath}
where the map $F: \bG(R) \to \bG(R)$ is the Frobenius endomorphism for the $\bk$-structure~$\bG_0$ (so that the one for  the $\bk_n$-structure $\bG_n$ is~$F^n$). 
Notice that projection on the first factor $(\bG^+)^F \to \bG^{F^n}$ is an isomorphism. We fix an $F^n$-stable pair~$(\bB, \bT)$ in $\bG$ and work with the $F$-stable pair $(\bB^+, \bT^+) = (\prod_{i=1}^n \bB, \prod_{i=1}^n \bT)$ in~$\bG^+$. Then there is an inclusion $\iota: W(\bT) \to W(\bT^+), w \mapsto (w, 1, \ldots, 1)$, inducing a bijection on $F$-conjugacy classes. Indeed, we see that $(w, 1, \ldots, 1)$ and $(xwF^n(x^{-1}), 1, \ldots, 1)$ are $F$-conjugates by $(x, F(x), \ldots, F^{n-1}(x))$, and given an arbitrary $x = (x_1, \ldots, x_n)$ we can always find~$\alpha = (\alpha_1, \ldots, \alpha_n)$ such that
$\alpha x F(\alpha)^{-1}$ is in the image of~$\iota$: it suffices to choose~$\alpha_1$ arbitrarily and to solve the equations $\alpha_i x_i F(\alpha_{i-1})^{-1} = 1$ recursively, for $2 \leq i \leq n$.

\begin{lemma}\label{Weilrestriction}
Let~$w \in W(\bT)$. There is an isomorphism $(\bT^+)^{\iota(w)F} \to \bT^{wF^n}$ identifying $R_w$ and $R_{\iota w}$.
\end{lemma}
\begin{proof}
The isomorphism is again projection on the first factor. Indeed, the fixed points in the target are given by $(t, F(t), \ldots, F^{n-1}(t))$ with the property that $t = wF^n(t)w^{-1}$. For the identification of Lusztig functors, we have that the cell $\bB^+ \iota(w) \bB^+$ decomposes as a product $\bB w \bB \times \bB \times \cdots \times \bB$, and so the preimage $\mL^{-1}(\bB^+ \iota(w) \bB^+)$ is given on $\cbk$-points by
\begin{displaymath}
(g, F(g)b_1, \ldots, F^{m-1}(g)b_{m-1})
\end{displaymath}
for arbitrary $b_i \in \bB(\cbk)$ and $g \in \bG(\cbk)$ such that $g^{-1}F^m(g) \in \bB w \bB$. A similar calculation works for the unipotent groups, after choosing a representative~$\dot{w}$ of~$w$ and the corresponding representative~$(\dot{w}, 1 \ldots, 1)$ of~$\iota(w)$. It follows that projection onto the first component induces a bijection $\bY(\dot{w}, 1, \ldots, 1) \to \bY(\dot{w})$, which is equivariant with respect to our isomorphisms $(\bG^+)^F \to \bG^{F^n}$ and $(\bT^+)^{\iota(w)F} \to \bT^{wF^n}$.
\end{proof}

\begin{lemma}\label{products}
For~$i \in \{1, \ldots, n \}$, fix connected reductive groups~$\bG_{0, i}$ over~$\bk$, pairs~$(\bB_i, \bT_i)$ in~$\bG_i$, and elements~$w_i \in W(\bT_i)$. Let~$\bG_0 = \prod_i \bG_{0, i}$ with~$(\bB, \bT) = (\prod_i \bB_i, \prod_i, \bT_i)$, and $\dot{w} = (\dot{w}_1, \cdots, \dot{w}_n)$. 
Then $R_w: R_{\cbQ_\ell}(\prod_i \bT_i^F) \to R_{\cbQ_\ell}(\prod_i \bG_i^F)$ sends a one-dimensional character~$\chi_1\cdots\chi_n$ to $R_{w_1}(\chi_1)\cdots R_{w_n}(\chi_n)$.
\end{lemma}
\begin{proof}
As in the proof of Lemma~\ref{Weilrestriction} we have an equivariant bijection $\bY(\dot{w}) \to \prod_i \bY(\dot{w}_i)$, and the claim follows from the K\"unneth formula for the cohomology of~$\bY(\dot{w})$.
\end{proof}

\paragraph{A character formula.} We now specialize to the case of~$\bG_0 = \GL_{n, \bk}$, with~$\bB$ the upper triangular Borel subgroup and $\bT$ the diagonal torus. The Weyl group~$W(\bT)$ identifies with the symmetric group~$S_n$, the $F$-conjugacy classes coincide with the conjugacy classes, and we normalize the lifts~$\dot{w}$ via permutation matrices. We give a formula for the Lusztig induction map corresponding to the Weyl group element~$w = (1, 2, \ldots, n)$, on semisimple conjugacy classes. The group~$\bT^{wF}$ is isomorphic to~$\bk_n^\times$. Choosing a basis of~$\bk_n$ as a $\bk$-vector space yields an inclusion $\Res_{\bk_n/\bk}\bbG_m \to \GL_{n, \bk}$ contained in the $\bG^F$-conjugacy class of rational maximal tori classified by~$w$. These tori all have the same $\epsilon$-invariant, which we will denote by~$\epsilon_w$. 
Notice that in our case the signs $\epsilon_{\bG_0} = (-1)^n$ and~$\epsilon_w = -1$, but it will sometimes be convenient not to make them explicit.
The following proposition is a very special case of the Lusztig classification we will discuss later, and reformulates the Green parametrization of cuspidal representations in terms of Deligne--Lusztig induction.

\begin{pp}\label{Greenparametrization}
Let~$\chi: \bk_n^\times \to \cbQ_l^\times$ be a $\Gal(\bk_n/\bk)$-regular character. 
Then the function $(-1)^{n-1} R_w(\chi)$ is the character of an irreducible cuspidal representation of~$\GL_n(\bk)$.
The map~$\chi \mapsto (-1)^{n-1} R_w(\chi)$ induces a bijection from the set of orbits of~$\Gal(\bk_n/\bk)$ on the $\bk$-regular characters of~$\bk_n^\times$, to the irreducible cuspidal representations of $\GL_n(\bk)$ over~$\cbQ_l$.
\end{pp}

\begin{rk}\label{DLcoefficients}
If~$\chi$ is not regular then it is not always the case that~$(-1)^{n-1}R_w(\chi)$ is effective.
However, these virtual representations will be important for us, since they will give rise to the Breuil--M\'ezard cycles of discrete series deformation rings.
\end{rk}

In the next proposition, we compute the character~$(-1)^{n-1} R_w(\chi)$ on semisimple classes, generalizing a well-known calculation in the case of $\bk$-regular~$\chi$.

\begin{pp}\label{DLinduction}
Let $\chi: \bk_n^\times \to \cbQ_\ell^\times$ be a character, and let~$w$ be the conjugacy class of $n$-cycles in the symmetric group~$S_n$. 
Then $R_w(\chi)$ vanishes on semisimple conjugacy classes of~$\GL_n(\bk)$ not represented in~$\bk_n^\times$, and for~$x \in \bk_n^\times$ we have
\begin{displaymath}
\epsilon_{\bG_0}\epsilon_{w}R_w(\chi)(x) = (-1)^{n+n/\deg(x)}(\GL_{n/\deg(x)}(\bk_{\deg(x)}) : \bk_n^\times)_{p'} \sum_{\gamma \in \Gal(\bk_{\deg(x)}/\bk)} \chi(\gamma x)
\end{displaymath}
where $\deg(x)$ is the degree of~$x$ over~$\bk$.
\end{pp}
\begin{proof}
By~\cite[Proposition~7.5.4]{Carterbook}, we have the equality
\begin{displaymath}
\epsilon_{\bG_0}\epsilon_w R_w(\chi) \St_{\bG_0} = \Ind_{\bk_n^\times}^{\GL_n(\bk)}(\chi)
\end{displaymath}
where~$\St_{\bG_0}$ is the Steinberg character and the induction is taken with respect to the embedding of~$\bk_n^\times$ in~$\GL_n(\bk)$ corresponding to some $\bk$-basis of~$\bk_n$. By~\cite[9.3~Corollary]{DMbook}, the Steinberg character vanishes away from semisimple classes, and if~$x$ is a semisimple element of~$\GL_n(\bk)$ then
\begin{displaymath}
\St_{\bG_0}(x) = \epsilon_{\bG_0}\epsilon_{Z_{\bG}^+(x)}|Z_{\bG^F}^+(x)|_p
\end{displaymath}
where~$Z_{\bG}^+(x)$ is the centralizer of~$x$, a connected reductive group over~$\bk$. 

Hence $R_w(\chi)(x) = 0$ if~$x$ is a semisimple element with no conjugates in~$\bk_n^\times$. When $x \in \bk_n^\times$, we compute the character of the induction as
\begin{align*}
\Ind_{\bk_n^\times}^{\GL_n(\bk)}(\chi)(x) = |Z_{\bG^F}^+(x)||\bk_n^\times|^{-1}\sum_{\gamma \in \Gal(\bk_{\deg(x)}/\bk)}\chi(\gamma x)
\end{align*}
since the $\bG^F$-conjugates of~$x$ in~$\bk_n^\times$ are precisely its Galois conjugates. The centralizer is isomorphic to~$\GL_{n/\deg(x)}(\bk_{\deg(x)})$. Then the claim follows since~$\Res_{\bk_{\deg(x)}/\bk}\bbG_m^{\times n/\deg(x)}$ is a quasisplit maximal torus in~$\Res_{\bk_{\deg(x)}/\bk}\GL_{n/\deg(x)}$ of rational rank~$\frac{n}{\deg(x)}$.
\end{proof}

\begin{rk}\label{DLinductiongeneralcoefficients}
For any field~$R$ of characteristic zero containing all roots of unity of order dividing the exponent of~$\GL_n(\bk)$ there exists a unique map~$\chi \mapsto (-1)^{n-1} R_w(\chi)$, from $R$-characters of~$\bk_n^\times$ to virtual $R$-representations of~$\GL_n(\bk)$, that satisfies the same character identity as~\cite[Theorem~4.2]{DLreps}.
It induces a bijection from regular $R$-characters to irreducible supercuspidal $R$-representations, which is already characterized by the formula in Proposition~\ref{DLinduction}, because of~\cite[Theorem~1.1]{SZcharacters}.
If~$\chi$ is an $R$-character of~$\GL_n(\bk)$, we will sometimes abuse notation and refer to~$(-1)^{n-1} R_w(\chi)$ as the Deligne--Lusztig induction of~$\chi$, even if strictly speaking we are not repeating the same construction using cohomology with $R$-coefficients.
(As a special case, this applies to~$R = \cbQ_p$.)
\end{rk}

\paragraph{Unipotent characters.}
Let~$\chi: W(\bT) \to \cbQ_l$ be the character of an irreducible representation. 
By~\cite[Theorem~15.8]{DMbook}, the unipotent characters of~$\bG^F$ are the functions
\begin{displaymath}
A_\chi = |W(\bT)|^{-1}\sum_{w \in W(\bT)} \chi(w)R_w(1_{\bT^{wF}})
\end{displaymath}
for varying~$\chi$. Notice that the maps~$R_{w_i}$ for $w_2 = ww_1w^{-1}$ are intertwined by the isomorphism $\ad(\dot{w}) : \bT^{w_1 F} \to \bT^{w_2 F}$, since for an arbitrary $F$-stable maximal torus~$\bT$ the map~$R_{\bT \subset \bB}^\bG$ does not depend on the choice of Borel subgroup containing~$\bT$ (see~\cite[Corollary~11.15]{DMbook}). By orthogonality of Deligne--Lusztig characters we deduce that
\begin{equation}\label{unipotentdevelopment}
(R_w(1_{\bT^{wF}}), A_\chi)_{\bG^F} = \chi(w),
\end{equation}
and so
\begin{displaymath}
R_w(1_{\bT^{wF}}) = \sum_{\chi \in \Irr(W(\bT))}\chi(w)A_\chi
\end{displaymath}
since the unipotent characters form an orthonormal family. 
Since~$R_{\id_W}$ coincides with the parabolic induction from~$\bT^F$, we see that the unipotent characters are the characters of the irreducible representations with supercuspidal support~$n\cdot 1_{\bk^\times}$. Further, by \cite[Proposition~12.13]{DMbook} we have that~$A_{\triv}$ is the trivial character of~$\bG^F$. 
It follows from our discussion of Harish-Chandra series that $\sigma_{\fP_{\min}}(1) = \St(1, n)$ is the only other factor of~$R_\id(\triv)$ with multiplicity one. 
This is~$A_{\sgn}$, where~$\sgn: W(\bT) \to \cbQ_l^\times$ is the sign character. 

\paragraph{Lusztig series.} Recall that two pairs $(\bT_i, \theta_i)$ consisting of a rational maximal torus in~$\bG$ and a character of~$\bT_i^F$ are said to be \emph{geometrically conjugate} if there exists $g \in \bG(\cbk)$ such that $\bT_2 = \ad(g)\bT_1$ and, for all~$n$ such that $F^n(g) = g$, we have 
\[
\theta_1 \circ N_{\bk_n/\bk} = \theta_2 \circ N_{\bk_n/\bk} \circ \ad(g).
\]
Here, the norm of an $F$-stable torus~$\bS$ is defined to be the morphism 
\[
N_{\bk_n/\bk}: \bS \to \bS, t \mapsto tF(t)\cdots F^{n-1}(t), 
\]
and we are asking for equality to hold on~$\bT_i^{F^n}$. 
For~$w \in S_n$, we write~$N_w$ for the $\bk_n/\bk$-norm of the diagonal torus with Frobenius endomorphism~$wF$.

By Langlands duality, one can construct a bijection between geometric conjugacy classes of pairs~$(\bS, \theta)$ in~$\bG$ and semisimple conjugacy classes in~$\bG^F \cong \GL_n(\bk)$.
The details are in~\cite[Chapter~13]{DMbook}.
Here we just remark that the construction depends on a choice of norm-compatible generators~$\zeta_n$ of every~$\bk_n^\times$, and an embedding $\cbk^\times \to \cbQ_l^\times$.

\begin{example}\label{classes} The geometric conjugacy class of a pair~$(\bS, \theta)$ such that~$\bS$ has type $w = (1, 2, \ldots, n)$ corresponds to the semisimple conjugacy class in~$\GL_n(\bk)$ whose characteristic polynomial is a power of the minimal polynomial of~$\theta(\zeta_n)$ over~$\bk$.
\end{example}

By~\cite[Proposition~13.3, Theorem~14.51]{DMbook}, two virtual characters $R_{w_i}(\theta_i)$ admit a common constituent if and only if the $(\bT_{w_i}, \theta_i)$ are geometrically conjugate; and furthermore, by~\cite[Proposition~13.1]{DMbook}, every irreducible character of~$\bG^F$ is a constituent of some~$R_w(\theta)$. It follows that the geometric conjugacy classes partition the set of irreducible characters of~$\bG^F$. 
An equivalence class in this partition is called the \emph{Lusztig series} of the corresponding semisimple conjugacy class~$s$ in~$\bG^F$, and it is denoted~$\mE(\bG^F, s)$. 
The unipotent characters form the Lusztig series $\mE(\bG^F, [1])$. 
We record the following theorem, which implies that in certain cases Lusztig induction preserves irreducibility. 
We will apply it in the next paragraph.

\begin{thm}\cite[Theorem~13.25]{DMbook}\label{Lusztigseries}
Let~$s$ be a semisimple element of~$\GL_n(\bk)$, and let~$\bL$ be a rational Levi subgroup of~$\bG$ containing the centralizer $Z_{\bG}(s)$. Then the map $\epsilon_{\bG}\epsilon_{\bL}R_{\bL}^{\bG}$ (taken with respect to any parabolic~$\bP$ with Levi factor~$\bL$) induces a bijection $\mE(\bL^F, [s]_{\bL^F}) \to \mE(\bG^F, [s]_{\bG^F})$.
\end{thm}

\paragraph{Virtual representations.} 
Let~$m$ be a positive divisor of~$n$ and let~$\pi_0$ be an irreducible supercuspidal representation of~$\GL_m(\bk)$. 
Since the matrix of Kostka numbers is upper unitriangular, it follows from the structure of the~$\pi_{\fP}(\pi_0)$ that they form a basis for the Grothendieck group of finite length representations of~$\GL_n(\bk)$ all of whose factors have supercuspidal support~$(n/m)\pi_0$. 
Then for any partition~$\fP$ of~$n/m$ there exists an element~$\sigma^+_{\fP}(\pi_0)$ of this Grothendieck group such that 
\begin{align*}
(\sigma^+_{\fP}(\pi_0), \pi_{\fP'}(\pi_0))_{\GL_n(\bk)} & = 1 \; \text{if} \; \fP = \fP' \\
& = 0 \; \text{otherwise}.
\end{align*}
We now give an explicit construction of~$\sigma^+_{\fP_{\min}}(\pi_0)$ in terms of Deligne--Lusztig theory.

\begin{thm}\label{sigmamax}
Let~$w$ be an $n$-cycle in~$S_n$, let~$w_m$ be an $m$-cycle in~$S_m$, and assume~$\pi_0 \cong (-1)^{m+1} R_{w_m}(\chi)$ for a $\bk$-regular character $\theta_0 : \bk_m^\times \to \cbQ_l^\times$. 
Let~$\theta = N_{\bk_n/\bk_m}^*(\theta_0)$. 
Then 
\[
\sigma^+_{\fP_{\min}}(\pi_0) \cong (-1)^{n+1}R_w(\theta).
\]
\end{thm}
\begin{proof}
First observe that $R_w(\theta)$ is orthogonal to each of the~$\pi_{\fP}(\pi_0)$ for~$\fP \not = \fP_{\min}$, because these are full parabolic inductions, and the torus~$\bT_w$ has no conjugates in any proper split Levi subgroup of~$\bG$.
So we have to prove that all irreducible constituents of~$R_w(\theta)$ have supercuspidal support $(n/m)\pi_0$, and that
\[
(R_w(\theta), \St(\pi_0, n/m))_{\bG^F} = (-1)^{n+1},
\]
since $\pi_{\fP_{\min}}(\pi_0) = \sigma_{\fP_{\min}}(\pi_0) = \St(\pi_0, n/m)$.

Write~$z = \theta_0(\zeta_m)$, so that the geometric conjugacy class of~$(\bT_w, \theta)$ corresponds to the minimal polynomial of~$z$ over~$\bk$ (a degree~$m$ polynomial) to the $n/m$-th power, as in Example~\ref{classes}. 
The centralizer in~$\GL_n(\bk)$ of any rational element in this conjugacy class is isomorphic to~$\GL_{n/m}(\bk_m)$, and it is the group of rational points of a Levi subgroup~$\bL_0$ of~$\bG_0$, isomorphic to $\Res_{\bk_m/\bk}\GL_{n/m, \bk_m}$. 
By the discussion preceding Lemma~\ref{Weilrestriction}, the conjugacy classes of rational maximal tori in~$\bL_0 \times_{\bk} \cbk$ are in bijection with those in~$\GL_{n/m, \bk_m} \times_{\bk_m} \cbk$. 
Under this bijection, the torus~$\bT_w$ has type corresponding to the $n/m$-cycles, which we write as~$w_{n/m}$.

By Lemma~\ref{Weilrestriction}, the unipotent characters~$\mE(\bL^F, [1])$ coincide with the unipotent characters of~$\GL_{n/m}(\bk_m)$ viewed as the group of $\bk_m$-points of~$\GL_{n/m, \bk_m}$.
Hence they are parametrized by~$\chi \in \Irr(S_{n/m})$ as in our previous discussion: we write~$\chi \mapsto A_\chi$ for this parametrization.

\begin{lemma}\label{explicitLusztigseries}
The Lusztig series of the geometric conjugacy class of~$(\bT_w, \theta)$ is 
\[
\{(-1)^{n+n/m}R_\bL^\bG(\theta_0 A_\chi): \chi \in \Irr(S_{n/m})\}.
\]
\end{lemma}
\begin{proof}
The character~$\theta_0$ can be inflated to~$\GL_{n/m}(\bk_m)$ via the determinant, and its restriction to~$\bk_n^\times$ is~$\theta$.
By~\cite[Proposition~13.30]{DMbook}, the Lusztig series of~$\bL^F$ attached to the geometric conjugacy class of~$(\bT_w, \theta)$ consists of the twists by~$\theta$ of the unipotent characters of~$\bL^F$.
Then the lemma follows from Theorem~\ref{Lusztigseries}.
\end{proof}

Theorem~\ref{sigmamax} will follow from Lemma~\ref{explicitLusztigseries}, Lemma~\ref{twoseries}, and the following two equations:
\begin{align}
R_w(\theta) = \sum_{\chi \in \Irr(S_{n/m})} \chi(w_{n/m}) R_\bL^\bG(\theta_0 A_\chi)\label{firsttoprove}\\
\St(\pi_0, n/m) = (-1)^{n + n/m} R_\bL^\bG(\theta_0 A_\sgn)\label{secondtoprove}.
\end{align}
Indeed, they imply that
\begin{align*}
(R_w(\theta), \sigma_{\fP_{\min}}(\pi_0))_{\bG^F}  &= \left ( \sum_{\chi \in \Irr(S_{n/m})} \chi(w_{n/m}) R_{\bL}^{\bG}(\theta_0 A_\chi), (-1)^{n+n/m}R_{\bL}^{\bG}(\theta_0 A_{\sgn}) \right )_{\bG^F}\\ 
&= (-1)^{n+n/m}\sgn(w_{n/m}) = (-1)^{n+1}
\end{align*}
since~$\sgn(w_{n/m}) = (-1)^{n/m+1}$.

\begin{proof}[Proof of Equation~(\ref{firsttoprove})]
Transitivity of Lusztig induction (see~\cite[11.5]{DMbook}) implies that $R_w(\theta) = R_{\bL}^{\bG}(R_{\bT_w}^{\bL}(\theta))$, where we have chosen an arbitrary parabolic subgroup~$\bP \subseteq \bG$ with Levi factor~$\bL = \bL_0 \times_{\bk} \cbk$. 
By~\cite[Corollary~1.27]{DLreps}, we have an equality $R_{\bT_w}^{\bL}(\theta) = \theta_0 R_{\bT_w}^{\bL}(1_{\bT_w})$. 
By Lemma~\ref{Weilrestriction}, the functor~$R_{\bT_w}^\bL$ coincides with~$R_{w_{n/m}}$ taken with respect to~$\GL_{n/m, \bk_m}$.
But we have seen in~(\ref{unipotentdevelopment}) that 
\begin{equation}\label{trivialcharactertwist}
R_{w_{n/m}}(1) = \sum_{\chi \in \Irr(S_{n/m})} \chi(w_{n/m}) A_\chi.
\end{equation}
\end{proof}

\begin{lemma}\label{twoseries}
The Lusztig series of~$(\bT_w, \theta)$ coincides with the Harish-Chandra series of~$n/m \cdot \pi_0$.
\end{lemma}
\begin{proof}
Let~$w_{m, n/m} \in S_n$ be the product of~$n/m$ disjoint $m$-cycles and let~$w_m \in S_m$ be an $m$-cycle.
Then the group of rational points of the torus~$\bT_{w_{m, n/m}}$ is isomorphic to $(\bk_m)^{\times n/m}$.
We are going to prove the lemma by computing
\[
(-1)^{n + n/m}R_{\bT_{w_{m, n/m}}}^{\bG}(\theta_0^{\otimes n/m})
\]
in two ways.
The first one is based on the fact that~$\bT_{w_{m, n/m}}$ is a rational maximal torus in~$\bL$: notice that $\bT_{w_{m, n/m}}(\bk)$ is the diagonal torus of~$\GL_{n/m}(\bk_m)$. 
Then similarly to~(\ref{trivialcharactertwist}), we have
\begin{equation}\label{trivialcharactersplit}
(-1)^{n + n/m}R_{\bT_{w_{m, n/m}}}^{\bG}(\theta_0^{\otimes n/m}) = (-1)^{n + n/m}R_{\bL}^{\bG}(\theta_0 R_\id(1)) = (-1)^{n + n/m}\sum_{\chi \in \Irr(S_{n/m})} \chi(\id) R_{\bL}^{\bG}(\theta_0 A_\chi)
\end{equation}
and so, by Lemma~\ref{explicitLusztigseries}, the constituents of this character coincide with the Lusztig series of~$(\bT_w, \theta)$

On the other hand, let $\bM = \GL_{m, \bk}^{\times n/m}$, a split Levi subgroup of~$\bG$.
Notice that~$\bT_{w_{m, n/m}}$ is a maximal torus in~$\bM$ (indeed, $\bk_m$ is a maximal torus in~$\GL_m(\bk)$).
By transitivity, $R_{\bT_{w_{m, n/m}}}^{\bG}(\theta_0^{\otimes n/m})$ is the character of the parabolic induction of~$R_{\bT_{w_{m, n/m}}}^{\bM}(\theta_0^{\otimes n/m})$, because Lusztig induction from a split Levi subgroup coincides with parabolic induction~\cite[11.1]{DMbook}. 
Now we can apply Lemma~\ref{products} to find that~$R_{\bT_{w_{m, n/m}}}^{\bM}(\theta_0^{\otimes n/m})$ equals~$R_{\bT_{w_m}}^{\GL_{m, \bk}}(\theta_0)^{\otimes n/m}$. Finally, we deduce that
\begin{align*}
(-1)^{n+n/m}R_{\bT_{w_{m, n/m}}}^{\bG}(\theta_0^{\otimes n/m}) = \operatorname{PInd}_{\bM^F}^{\bG^F}\left ( (-1)^{m+1} R_{w_m}(\theta_0) \right )^{\otimes n/m} & = \operatorname{PInd}_{\prod_{i=1}^{n/m}\GL_m(\bk)}^{\GL_n(\bk)}(\pi_0^{\otimes n/m}) \\
& = \pi_{\fP_{\max}}(\pi_0).
\end{align*}
Then the lemma follows from~(\ref{trivialcharactersplit}) and the fact that the Harish-Chandra series of~$n/m \cdot \pi_0$ coincides with the set of constituents of~$\pi_{\fP_{\max}}(\pi_0)$.
\end{proof}

\begin{proof}[Proof of Equation~(\ref{secondtoprove})]
The character~$A_{\sgn}$ is the Steinberg character of~$\GL_{n/m}(\bk_m)$, and the Lusztig induction of a nondegenerate character is nondegenerate, by~\cite[Th{\'e}or{\`e}me~4.4]{DMLusztigGreen} (the nondegenerate irreducible characters are the constituents of a Gelfand--Graev representation). 
Since nondegeneracy is preserved under twisting by one-dimensional characters (because unipotent elements have determinant one), we see that $(-1)^{n+n/m} R_{\bL}^{\bG}(\theta_0 A_{\sgn})$ is the character of a nondegenerate representation in the Lusztig series of~$(\bT_w, \theta)$.
By Lemma~\ref{twoseries}, this representation is~$\St(\pi_0, n/m)$.
\end{proof}

This completes the proof of Theorem~\ref{sigmamax}.
\end{proof}

\section{Type theory.}\label{typetheory}
In this section we recall the structure of maximal simple types for the inner forms of~$\GL_n(F)$ and the results of Schneider and Zink about $\bK$-types. We establish their analogues for types on the maximal compact subgroup of~$D^\times$. We then prove some formulas for the trace of a $\bK$-type in terms of its level zero part and its base change to unramified extensions, and begin studying their behaviour under the Jacquet--Langlands correspondence. From now on, whenever dealing with a finite general linear group~$\GL_n(\bF_q)$ we will write~$R_w(-)$ for the Deligne--Lusztig induction from a maximal torus whose type consists of the $n$-cycles. 
If~$\chi$ is an $\bF_q$-regular character of~$\bF_{q^n}^\times$ then $(-1)^{n-1}R_w(\chi)$ is an irreducible cuspidal representation of~$\GL_n(\bF_q)$.
(See Remark~\ref{DLinductiongeneralcoefficients}.)
We will also fix the maximal compact subgroup $\bK = \GL_n(\mO_F)$ of~$\GL_n(F)$ for the rest of the paper.

\paragraph{Maximal simple types.} In this paragraph we let~$G = \GL_m(D)$ be an inner form of~$\GL_n(F)$, for~$D$ a central division algebra over~$F$ of reduced degree~$d$. We write~$A = M_m(D)$. 
We summarize the parametrization of simple inertial classes of representations of~$G$ from the point of view of~\cite{InertialJL}, which builds upon the work of Bushnell--Kutzko~\cite{BKbook} and Broussous, S\'echerre and Stevens in a series of papers (see for instance \cite{BSSV} and~\cite{SecherreStevensJL}). 
Recall that the simple inertial classes of representations of~$G$ are those whose supercuspidal support is inertially equivalent to~$r\pi_0$ for some positive divisor $r|m$ and some supercuspidal representation~$\pi_0$ of~$\GL_{m/r}(D)$. 

The supercuspidal Bernstein components of~$G$ admit types constructed as follows. 
One starts with a \emph{maximal simple character}, which is a character of a compact open subgroup~$H^1_\theta$ of~$G$. 
As~$m$ and~$D$ vary, the maximal simple characters of~$\GL_m(D)$ can be classified according to their endo-class (usually denoted~$\Theta_F$). 
Two maximal simple characters in~$G$ have the same endo-class if and only if they are $G$-conjugate. 

Attached to a maximal simple character~$\theta$ there are subgroups $H^1_\theta \subseteq J^1_\theta \subseteq J_\theta$, of~$G$, each normal in the next. 
There corresponds to~$\theta$ an irreducible representation~$\eta_\theta$ of~$J^1_\theta$, whose restriction to~$H^1_\theta$ is a multiple of~$\theta$. 
One can extend~$\eta_\theta$ to a representation of~$J_\theta$, and a class of \emph{$\beta$-extensions} is singled out. 
They are all twists of each other by characters of~$J_\theta/J^1_\theta$, which is a finite general linear group.
There exists a unique $\beta$-extension~$\kappa_p$, called the $p$-primary $\beta$-extension, with the property that the order of the character $\det(\kappa_p)$ is a power of~$p$.
However, the main result of~\cite{InertialJL} suggests that we work instead with a certain quadratic twist~$\kappa_\theta = \epsilon^1_\theta \kappa_p$. 
The character~$\epsilon^1_\theta$ is one of the ``symplectic invariants'' of~$\theta$, for which see~\cite[Propositions 2.11, 2.13]{InertialJL} and the references therein (note that~$\epsilon^1_\theta$ is denoted~$\epsilon^1(-, V_\theta)$ in~\cite{InertialJL}).

To go further in the construction, let $\theta$ be a maximal simple character with endo-class $\cl(\theta) = \Theta_F$.
We identify the group~$J_\theta/J^1_\theta$ with a certain finite general linear group. 
As in \cite{InertialJL}, we let $E/F$ be the unramified parameter field of~$\Theta_F$ in~$\overline{F}$. If~$[\fA, \beta]$ is a simple stratum for~$\theta$, and $Z_A(F[\beta]) \cong M_{m'}(D')$ for a central division algebra~$D'$ over~$F[\beta]$ of reduced degree~$d'$, then $m'd' = n/\delta(\Theta_F)$, and~$J_\theta/J^1_\theta \cong \GL_{m'}(\be_{d'})$ (we recall that $\be_{d'}$ is the residue field of the unramified extension of~$E$ in~$\overline{F}$ of degree~$d'$). 
More precisely, in~\cite[Section~3.1]{InertialJL} there is constructed an injection from the set of lifts of~$\Theta_F$ to an endo-class~$\Theta_E$ over~$E$ to the set of conjugacy classes of isomorphisms
\begin{displaymath}
J_{\theta} / J^1_{\theta} \to \GL_{m'}(\be_{d'})
\end{displaymath}
under the group $\GL_{m'}(\be_{d'}) \rtimes \Gal(\be_{d'}/\be)$.
(The notion of lift of an endo-class over~$F$ to a finite tamely ramified extension of~$F$ is defined in~\cite[Section~9]{BHliftingI}, see especially~\cite[Corollary~9.13]{BHliftingI}.)

Let $\chi: \be_{n/\delta(\Theta_F)}^\times \to \bC^\times$ be a $\be_{d'}$-regular character, that is to say a character with trivial stabilizer in $\Gal(\be_{n/\delta(\Theta_F)}/\be_{d'})$. Then the Deligne--Lusztig induction~$(-1)^{m'+1}R_{w}(\chi)$ is a supercuspidal representation of~$\GL_{m'}(\be_{d'})$, depending only on the $\Gal(\be_{n/\delta(\Theta_F)}/\be_{d'})$-conjugacy class of~$\chi$. The representation $(J_\theta, \kappa_\theta \otimes (-1)^{m'+1}R_{w}(\chi))$ is a type for a supercuspidal Bernstein component of~$G$, and all such components admit types of this kind. Furthermore, two types $(J_\theta, \kappa_\theta \otimes (-1)^{m'+1}R_{w}(\chi_i))$ determine the same component if and only if the~$\chi_i$ are conjugate under $\Gal(\be_{n/\delta(\Theta_F)}/\be)$.

The choice of~$\Theta_E \to \Theta_F$ and of the $\beta$-extension~$\kappa_\theta$ determines a \emph{level zero map},
denoted~$\Lambda(-, \Theta_E, \kappa_\theta)$ in~\cite{InertialJL}.
To shorten notation we will denote it~$\Lambda_{\kappa_\theta}(-)$ or simply~$\Lambda$, since most of the time we will be working with~$\kappa_\theta$ and with a fixed lift~$\Theta_E \to \Theta_F$.
It goes from the set of irreducible smooth representations of~$\GL_m(D)$ whose supercuspidal support is simple of endo-class~$\Theta_F$ to the set of $\Gal(\be_{n/\delta(\Theta_F)} / \be)$-orbits of characters of~$\be_{n/\delta(\Theta_F)}^\times$. 
It only depends on the inertial class of a representation, and it sends a supercuspidal representation~$\pi$ to the orbit~$[\chi]$ determined by the maximal simple types it contains. 
To describe its effect on simple, nonsupercuspidal representations, let~$\pi$ be an irreducible representation with supercuspidal support $r\pi_0$. 
Recall from~\cite{MStypes, InertialJL} that there exists a unique conjugacy class of maximal $\beta$-extensions in~$\GL_{m/r}(D)$ that is \emph{compatible} with~$\kappa_\theta$ in the sense explained in these references.
We denote it~$\kappa_\theta^0$.
Then~$\Lambda(\pi, \Theta_E, \kappa_\theta)$ is the inflation to~$\be_{n/\delta(\Theta_F)}^\times$ of $\Lambda(\pi_0, \Theta_E, \kappa_\theta^0)$.

In summary, if we fix a lift~$\Theta_E \to \Theta_F$ for all endo-classes~$\Theta_F$ over~$F$ we find a bijection from the set of simple inertial classes of~$G$ to the set of pairs $(\Theta_F, [\chi])$ consisting of an endo-class over~$F$ of degree dividing~$n$ and an orbit of~$\Gal(\be_{n/\delta(\Theta_F)}/\be)$ on the set of characters of~$\be_{n/\delta(\Theta_F)}^\times$. 

\begin{rk}
We treat the case of level zero representations by letting~$J_\theta$ be a maximal compact subgroup with principal congruence subgroup~$J^1_\theta$, letting $\kappa_\theta = 1$, and working with an arbitrary choice of isomorphism $J_\theta/J^1_\theta \to \GL_m(\mbf_d)$ extending to an $\mbf$-algebra isomorphism. This introduces no ambiguity as here $\be = \mbf$ and so the action of $\Gal(\mbf_d/\mbf)$ does not change the inertial class.
\end{rk}

Finally, we recall a special case of the interior lifting construction of~\cite{BHliftingI} and~\cite{BSSV}. 
If~$\theta$ is a maximal simple character in~$\GL_m(D)$ and~$[\fA, \beta]$ is a simple stratum for~$\theta$, there exists a maximal unramified extension~$K^+/F[\beta]$ in~$Z_A(F[\beta])$ normalizing the order~$\fA$, as in the proof of~\cite[Proposition~2.5]{InertialJL}.
Let~$L$ be any unramified extension of~$F$ in~$K^+$. 
As in~\cite[Proposition~2.8, Lemma~2.12]{InertialJL}, the restriction $\theta_L = \theta|H^1_\theta \cap Z_A(L)$ is a maximal simple character with corresponding groups~$H^1_{\theta_L} = H^1_\theta \cap Z_A(L)$ and~$J^i_{\theta_L} = J^i_\theta \cap Z_A(L)$. 
The character~$\theta_L$ is called the interior lift of~$\theta$ to~$L$.

\paragraph{$\bK$-types for~$\GL_n(F)$.} 
Let~$A = M_n(F)$ and~$G = A^\times = \GL_n(F)$. 
We recall some results from~\cite{SZtypes} and translate them in the form we will need later on. 
Let $F[\beta]$ be a field extension of~$F$ in~$A = M_n(F)$, and let~$B = Z_A(F[\beta])$. 
Choose a pair $\fB_{\min} \subseteq \fB_{\max}$ of hereditary $\mO_{F[\beta]}$-orders in~$B$, such that~$\fB_{\min}$ is minimal and~$\fB_{\max}$ is maximal. 
Recall that hereditary $\mO_{F[\beta]}$-orders $\fB$ in~$B$ are in bijection with $\mO_{F[\beta]}$-lattice chains in~$V = F^n$ viewed as an $F[\beta]$-vector space via the inclusion $F[\beta] \subset A$. 
Since these are also $\mO_F$-lattice chains, there corresponds to~$\fB$ a unique hereditary $\mO_F$-order $\fA = \fA(\fB)$ of~$A$, called the continuation of~$\fB$ to~$A$.
It satisfies $\fA(\fB) \cap B = \fB$. 
Following~\cite[Section~5]{SZtypes}, we associate to~$\fB$ a subgroup $J = J(\fB) = J(\beta, \fA(\fB))$ of the unit group~$\fA(\fB)^\times$ such that $J = J^1\fB^\times$ for~$J^1 = J^1(\fB) = J \cap U^1(\fA(\fB))$. 
We write~$J_{\max}$ and~$J^1_{\max}$ for the groups corresponding to~$\fB_{\max}$.

\begin{rk}
From now on, we make the assumption that the group~$J_{\max}$ is contained in our fixed maximal compact subgroup~$\bK$. 
This can always be achieved after possibly replacing~$F[\beta]$ with a conjugate.
\end{rk}

We let~$\theta$ be a simple character of the stratum~$[\fA_{\max}, \beta]$ (so $\theta$ is  maximal) and write $J_{\max} = J_\theta$ and~$J^1_{\max} = J^1_\theta$.
We write~$\kappamax$ or $\kappa(\fB_{\max})$ for the $\beta$-extension~$\kappa_\theta$ (the paper~\cite{SZtypes} works with an arbitrary $\beta$-extension). 
There is a corresponding family of representations~$\kappa(\fB)$ of~$J(\beta)$, one for each for any hereditary $\mO_{F[\beta]}$-order $\fB_{\min} \subseteq \fB \subseteq \fB_{\max}$, satisfying a coherence property as in~\cite[Lemma~5.1]{SZtypes}.

Writing~$\Theta_F$ for the endo-class of~$\theta$, and~$E/F$ for the unramified parameter field of~$\Theta_F$ in~$\overline{F}$, we have attached an inner conjugacy class of isomorphisms
\begin{displaymath}
J_{\max} / J^1_{\max} \to \GL_{n/\delta(\Theta_F)}(\be)
\end{displaymath}
to every lift of~$\Theta_F$ to an endo-class~$\Theta_E$ over~$E$.
We have previously fixed such a lift $\Theta_E \to \Theta_F$, and we now let~$\psi$ be a representative of the corresponding conjugacy class such that~$\psi$ identifies $\fB_{\min}^\times J^1_{\max} / J^1_{\max}$ with the upper-triangular Borel subgroup (compare the discussion after~\cite[Lemma~5.5]{SZtypes}).

There is a functor~$V \mapsto V(\kappamax) = \Hom_{J^1_{\max}}(\kappamax, V)$, from the category of smooth representations of~$\GL_n(F)$ to the category of representations of~$J_{\max}/J^1_{\max}$, sending admissible representations to finite-dimensional ones, which we denoted~$\bK_{\kappamax}$ in~\cite{InertialJL}. 
We will compose it with our isomorphism~$\psi$, and denote the resulting functor still by~$V \mapsto V(\kappamax)$.

For any positive divisor~$r$ of~$n/\delta(\Theta_F)$ we have a standard parabolic subgroup of~$\GL_{n/\delta(\Theta_F)}(\be)$, with Levi factor isomorphic to~$\prod_{i=1}^r \GL_{n/r\delta(\Theta_F)}(\be)$, and consisting of block upper triangular matrices. 
It coincides with the image under~$\psi$ of~$\fB^\times J^1_{\max}/J^1_{\max}$ for some principal order $\fB_{\min} \subseteq \fB \subseteq \fB_{\max}$ that we fix. 
If~$\sigma_0$ is a cuspidal representation of~$\GL_{n/r\delta(\Theta_F)}(\be)$ attached to the character~$\chi$ of~$\be_{n/r\delta(\Theta_F)}^\times$, and~$\sigma = \sigma_0^{\otimes r}$ is inflated to~$J(\fB)/J^1(\fB)$, then the pair~$(J(\fB), \kappa(\fB) \otimes \sigma)$ is a simple type in~$\GL_n(F)$. 
It is a maximal simple type precisely when~$r=1$. 
The next lemma connects this construction with our parametrization of simple inertial classes.

\begin{lemma}\label{findclass}
Let~$\fs$ be the simple inertial class with invariants~$\cl(\fs) = \Theta_F$ and~$\Lambda(\fs, \Theta_E, \kappamax) = [\chi]$.
With the notation of the previous paragraph, the pair $(J(\fB), \kappa(\fB) \otimes \sigma)$ is a type for~$\fs$.
\end{lemma}
\begin{proof}
Let~$V$ be an irreducible simple representation of~$\GL_n(F)$ containing the simple type $(J(\fB), \kappa(\fB) \otimes \sigma)$.
Let the supercuspidal support of~$V$ be~$V_0^{\otimes s}$. 
Let~$\theta_0$ be a maximal simple character in~$\GL_{n/s}(F)$ of endo-class~$\Theta_F$. 
Let~$\kappamax^0$ be the $\beta$-extension of~$\theta_0$ compatible with~$\kappamax$. 
We need to prove that~$r = s$ and that $\Lambda(V_0, \Theta_E, \kappamax^0) = [\chi]$.

By definition of compatibility, the supercuspidal support of~$V(\kappamax)$ is a product of representations corresponding to $\Lambda(V_0, \Theta_E, \kappamax^0)$ under Deligne--Lusztig induction.
On the other hand, by~\cite[Proposition~5.3]{SZtypes}, the supercuspidal support of~$V(\kappamax)$ is~$\left [ \prod_{i=1}^r\GL_{n/r\delta(\Theta_F)}(\be), \sigma_0^{\otimes r} \right ]$. 
The lemma follows by uniqueness of supercuspidal support.
\end{proof}

Write~$\fs$ for the inertial class in Lemma~\ref{findclass}.
We are going to define two classes of virtual representations of~$\bK$ attached to~$\fs$, depending only on the maximal simple character~$\theta$.
Recall that we assume $J_\theta \subset \bK$.
If~$\fP$ is a partition of~$r$, the constructions of Section~\ref{finitegroups} provide us a representation $\kappa_{\max} \otimes \sigma_{\fP}(\sigma_0)$ of~$J_\theta$, via $J_\theta/J^1_\theta \cong \GL_{n/\delta(\Theta_F)}(\be)$.

\begin{defn}\label{definitionKtypes}
Write $\sigma_{\fP}(\fs) = \Ind_{J_\theta}^\bK(\kappamax \otimes \sigma_{\fP}(\sigma_0))$, which by the discussion at the end of~\cite[Section~5]{SZtypes} is an irreducible smooth representation of~$\bK$. 
Write~$\sigma_{\fP}^+(\fs)$ for the virtual representation~$\Ind_{J_\theta}^\bK(\kappamax \otimes \sigma_{\fP}^+(\sigma_0))$ of~$\bK$. 
We will refer to these representations as \emph{$\bK$-types} for~$\fs$.
\end{defn}

Next we show how the $\bK$-types provide a refinement of Bushnell--Kutzko type theory for generic representations.
Via the Bernstein--Zelevinsky classification, we can attach to each irreducible representation~$V \in \Irr(\fs)$ a partition~$\fP(V)$ of~$r$, in the following way.

\begin{defn}\label{partitionforV}
Let~$V \in \Irr(\fs)$. 
We define~$\fP(V)(i)$ to be the number of times a segment of length~$i$ appears in the multiset corresponding to~$V$.
We will sometimes shorten notation to~$\fP = \fP(V)$.
\end{defn}

\begin{pp}
Let~$V \in \Irr \GL_n(F)$ be generic. 
Let~$\fP$ be a partition of~$r$.
Then $\Hom_{J_\theta}(\kappa_{\max} \otimes \sigma_{\fP}(\sigma_0), V) \not = 0$ if and only if~$V \in \fs$ and its partition~$\fP(V)$ satisfies $\fP \leq \fP(V)$.
\end{pp}
\begin{proof}
By~\cite[Lemma~5.2]{SZtypes}, the nonvanishing implies that~$V \in \Irr(\fs)$.
By~\cite[Proposition~5.9]{SZtypes} we have that $V(\kappa_{\max}) \cong \pi_{\fP(V)}(\sigma_0)$ whenever~$V \in \Irr(\fs)$ is generic. 
Then the claim follows from the existence of an isomorphism
\begin{displaymath}
\Hom_{J_\theta}(\kappa_{\max} \otimes \sigma_{\fP}(\sigma_0), V) \isom \Hom_{\GL_{n/\delta(\Theta_F)}(\be)}(\sigma_{\fP}(\sigma_0), V(\kappa_{\max})).
\end{displaymath}
\end{proof}

\begin{rk}\label{nogenericityassumption}
By~\cite[Lemma~5.2]{SZtypes}, the fact that~$V \in \Irr(\fs)$ if $\Hom_{J_\theta}(\kappa_{\max} \otimes \sigma_{\fP}(\sigma_0), V) \not = 0$ holds for any $V \in \Irr \GL_n(F)$, with no genericity assumptions.
\end{rk}

\begin{example}\label{discreteseries}
Let~$V$ be irreducible and generic. We have $(\sigma^+_{\fP_{\min}}(\fs), V)_{\bK} \not = 0$ if and only if $V \in \Irr(\fs)$ and $V(\kappamax) \cong \pi_{\fP_{\min}}(\sigma_0)$, in which case it equals one. This happens if and only if~$\fP(V) = \fP_{\min}$, that is the multiset of~$V$ has only one segment (because~$\fP_{\min}$ is the partition with only one summand). Equivalently, $V$ is an essentially square-integrable representation in~$\fs$.
\end{example}

Finally, we remove the dependence of the $\bK$-types on~$\theta$.

\begin{pp}\label{thetaindependence}
The representations~$\sigma_{\fP}(\fs)$ and~$\sigma_{\fP}^+(\fs)$ are independent of the choice of~$\theta$.
\end{pp}
\begin{proof}
Let~$\theta_1$ and~$\theta_2$ be conjugate maximal simple characters in~$\GL_n(F)$, with $J_{\theta_i} \subseteq \bK$. 
The orders~$\fA_i$ attached to the~$\theta_i$ are principal orders with the same ramification, corresponding to lattice chains that contain the lattice chain defined by~$\bK$ (because~$\fA_i$ is the continuation of~$\fA_i \cap B_i$ and $(\fA_i \cap B_i)^\times \subseteq J_{\theta_i} \subseteq \bK$). 
Hence the~$\fA_i$ are $\bK$-conjugate. 
Since intertwining maximal simple characters defined on the same order are conjugate under the group of units of that order (see~\cite[Theorem~3.5.11]{BKbook}), we see that the~$\theta_i$ are conjugate under~$\bK$, hence so are the~$J_{\theta_i}$. 
Write $J_{\theta_2} = \ad(g)J_{\theta_1}$. 
Since the lift~$\Theta_E \to \Theta_F$ is fixed, by the proof of \cite[Proposition~3.9]{InertialJL} the inner conjugacy classes $[\psi_i]: J_{\theta_i}/J_{\theta_i}^1 \to \GL_{n/\delta(\Theta_F)}(\be)$ satisfy $[\psi_1] = \ad(g)^*[\psi_2]$. 
It follows that we get isomorphic representations when inducing.
\end{proof}

So there are well-defined representations $\sigma_{\fP}(\fs)$ and $\sigma_{\fP}^+(\fs)$ of~$\bK$ for every simple Bernstein component~$\fs$ of~$\GL_n(F)$. 
By Remark~\ref{nogenericityassumption}, these are \emph{typical} representations: each of them determines the Bernstein component of an irreducible representations of~$\GL_n(F)$ that contains it. 
We do not claim that these are the only typical representations of~$\bK$, although some variant of this statement (assuming~$p>n$) is expected to hold, as in~\cite[Conjecture~4.1.3]{EGBM}. 
This is closely related to the problem of ``uniqueness of types", for which see~\cite{Paskunasunicity} and~\cite[Annexe~A]{BMmultiplicites}.

\paragraph{$\bK$-types for~$D^\times$.} The group~$D^\times$ has a unique maximal compact subgroup~$\mO_D^\times$. Let~$(J_\theta, \lambda = \kappa_\theta \otimes \chi)$ be a maximal simple type in~$D^\times$, so that $J_{\theta} \subseteq \mO_D^\times$. Fix a simple stratum~$[\mO_D, \beta]$ for~$\theta$ and a uniformizer~$\pi_{D'}$ of the central division algebra $D' = Z_D(F[\beta])$ over~$F[\beta]$. Then the normalizer~$\bJ(\theta)$ of~$\theta$ in~$D^\times$ is~$\pi_{D'}^{\bZ} \ltimes J_\theta = (D')^\times J^1_\theta$, and the normalizer~$\bJ(\lambda)$ of~$\lambda$ in~$D^\times$ has index in~$\bJ(\theta)$ equal to the size~$b(\chi)$ of the orbit of~$\chi$ under~$\Gal(\be_{n/\delta(\Theta_F)}/\be)$ (see for instance~\cite[3.4]{MStypes}). 

By~\cite[Proposition~2.6.1]{BHJL}, the $D^\times$-intertwining set of~$(J_\theta, \lambda)$ coincides with~$\bJ(\lambda)$, which intersects~$\mO_D^\times$ in~$J_\theta$. It follows that the intertwining set of~$\lambda$ in~$\mO_D^\times$ is~$J_\theta$ and that $\Ind_{J_\theta}^{\mO_D^\times}\lambda$ is irreducible. By Frobenius reciprocity, it is a type for the Bernstein component corresponding to~$(J_\theta, \lambda)$. We will refer to these representations as \emph{$\bK$-types} for~$D^\times$.

Another construction of $\bK$-types in this context can be given as follows. A smooth irreducible representation~$\pi$ of~$D$ restricts to a semisimple representation of~$\mO_D^\times$, whose irreducible constituents form a unique orbit under conjugation by a uniformizer~$\Pi_D$ of~$D^\times$. By~\cite[Remark~1.6.1.3]{Rocheottawa}, each constituent occurs with multiplicity one. If~$\tau$ is another smooth irreducible representation of~$D^\times$, it follows that $\Hom_{\mO_D^\times}(\pi, \tau)$ is nonzero if and only if~$\pi|_{\mO_{D^\times}}$ and~$\tau|_{\mO_{D^\times}}$ are isomorphic, and this is equivalent to~$\pi$ and~$\tau$ being unramified twists of each other. It follows that any irreducible constituent of~$\pi|_{\mO_{D}^\times}$ is a type for the inertial class of~$\pi$. This is the construction used in~\cite{GGBM} when~$n = 2$.

In contrast with the case of~$\GL_n(F)$ (see \cite{Paskunasunicity}), the $\bK$-type of a supercuspidal representation need not be unique: by~\cite[Lemma~1.6.3.1]{Rocheottawa}, the number of constituents of~$\pi |_{\mO_{D}^\times}$ equals the \emph{torsion number} of~$\pi$, which is the number of unramified characters~$\chi$ of~$D^\times$ such that $\chi\pi \cong \pi$.

\paragraph{Trace formulas for $\bK$-types.} 
A conjugacy class in a profinite group~$G$ is \emph{pro-$p$-regular} if its elements are $p$-regular in all finite discrete quotients of~$G$ (that is, their order is coprime to~$p$). We have the following lemma.
\begin{lemma}\label{pregularconjugacy}
If ~$G$ is a profinite group, $H$ is a finite group, and $\pi: G \to H$ is a continuous surjection with pro-$p$ kernel, then~$\pi$ induces a bijection from the pro-$p$-regular classes of~$G$ to the $p$-regular classes of~$H$.
\end{lemma}
\begin{proof}
Assume that the claim is true when~$G$ is a finite group. 
Then every $p$-regular element~$h \in H$ has a $p$-regular lift in every finite discrete quotient of~$G$ surjecting onto~$H$.
Since a directed inverse limit of nonempty finite sets is nonempty, we find that~$h$ has a pro-$p$-regular lift in~$G$, and so the map induced by~$\pi$ is surjective.
To prove it is injective, let~$g_1$ and~$g_2$ be pro-$p$-regular elements of~$G$ that are conjugate in every finite discrete quotient~$\lbar G$ of~$G$.
Then the finite set of elements of~$\lbar G$ conjugating~$\lbar g_1$ to~$\lbar g_2$ is not empty.
Taking the inverse limit as~$\lbar G$ varies of these finite sets, we find an element of~$G$ that conjugates~$g_1$ to~$g_2$.
So the map induced by~$\pi$ is injective.

It follows that it suffices to prove the claim for~$G$ a finite group. In this case, the surjectivity of~$\pi$ on $p$-regular classes follows because every $p$-regular element~$\pi(x)$ of~$H$ admits a $p$-regular lift, since if $x = x_{(p)}x^{(p)}$ then the images of these under~$\pi$ commute, hence $\pi(x_{(p)}) = \pi(x)_{(p)} = 1$ and~$\pi(x^{(p)}) = \pi(x)$. Then the claim follows because~$G$ and~$H$ have the same number of $p$-regular classes, since every irreducible $\cbF_p$-representation of~$G$ is trivial on~$\ker(\pi)$, and the number of irreducible $\cbF_p$-representations of a finite group equals the number of its $p$-regular conjugacy classes~\cite[Section~18.2, Corollary~3]{Serrereps}.
\end{proof}

We now go back to the situation where~$G = \GL_n(F)$ for a finite extension~$F/\bQ_p$. 
Fix a maximal simple character~$\theta$ of endo-class~$\Theta_F$ and a character~$\chi: \be_{n/\delta(\Theta_F)}^\times \to \bC^\times$, possibly not $\be$-regular. 
Let~$[\fA, \beta]$ be a simple stratum for~$\theta$ and write $B = Z_A(F[\beta])$ and~$\fB = \fA \cap B$.
Let~$\fs_G$ be the unique simple inertial class in~$G$ with invariants
\[
\cl(\fs_G) = \Theta_F \text{ and } \Lambda(\fs_G, \Theta_E, \kappa_\theta) = [\chi].
\]
As in the above, we assume that $J_\theta \subseteq \bK$, we fix a maximal unramified extension~$K^+$ of~$F[\beta]$ in~$Z_A(F[\beta])$ normalizing~$\fA$, and we let~$K$ be the maximal unramified extension of~$F$ in~$K^+$. 
We remark that the unit group $K^\times$ normalizes~$\bK$: this is because $\mO_K \subseteq \mO_{K^+} \subseteq \fB$, and $K^\times = \pi_F^{\bZ} \times \mO_K^\times$ for a uniformizer~$\pi_F$ of~$F$, which is central in~$G$.

By Theorem~\ref{sigmamax}, the $\bK$-type $\sigma^+_{\fP_{\min}} \left ( \fs_G \right )$ is $\Ind_{J_\theta}^{\bK}(\kappa_\theta \otimes (-1)^{n/\delta(\Theta_F)+1}R_w(\chi))$, which is a virtual representation of~$\bK$ if~$\chi$ is not $\be$-regular. 
When~$\chi$ is $\be$-regular, this is a maximal simple type and~$\fs_G$ is a supercuspidal inertial class.
We will shorten notation to $\lambda = \kappa_\theta \otimes (-1)^{n/\delta(\Theta_F)+1}R_w(\chi)$ and $\sigma^+ = \sigma^+_{\fP_{\min}} \left ( \fs_G \right )$.

\begin{pp}\label{ellipticvanishing}
If $x \in \bK$ is a pro-$p$-regular element that is not $\bK$-conjugate to an element of~$\mu_K$ then $\tr \sigma^+(x) = 0$.
(Recall that $\mu_K$ is the group of prime-to-$p$ roots of unity in~$K^\times$.)
\end{pp}
\begin{proof}
By the Frobenius formula for an induced character we have
\begin{displaymath}
\tr \sigma^+(x) = \sum_{y \in J_\theta \backslash \bK} \tr \lambda(yxy^{-1}).
\end{displaymath}
By Lemma~\ref{pregularconjugacy}, the pro-$p$-regular conjugacy classes of~$J_\theta$ are in bijection with the semisimple conjugacy classes of~$\GL_{n/\delta(\Theta_F)}(\be)$, via our isomorphism $J_\theta/J^1_\theta \to \GL_{n/\delta(\Theta_F)}(\be)$. 
Now the claim follows by Proposition~\ref{DLinduction}, as $R_w(\chi)$ vanishes on semisimple conjugacy classes that are not represented in a maximal elliptic torus, and~$\mu_K = \mu_{K^+}$ maps isomorphically to such a torus.
\end{proof}

We now give a formula for~$\tr \sigma^+(x)$ when $x \in \mu_K$ generates an unramified extension~$L/F$ (in the sense that~$L = F[x]$). 
For this, we take the interior lift~$\theta_L$, and notice the decomposition
\begin{displaymath}
J_{\theta_L} = \GL_{n/[L[\beta]:F]}(\mO_{L[\beta]})J^1_{\theta_L}.
\end{displaymath}We write~$G_L = Z_G(L)$ and notice the equality $Z_\bK(L) = \bK \cap G_L$. 
Since~$K^\times$ normalizes~$\bK$, this is a maximal compact subgroup of~$G_L$ that we denote~$\bK_L$.

We are going to apply the Glauberman correspondence in the form stated in~\cite{BHLLIII, InertialJL}. 
Let $\widetilde{\eta}(\theta)$ be the only extension of~$\eta_\theta$ to~$\mu_K \ltimes J^1_\theta$ whose determinant is trivial on~$\mu_K$. 
Then~$\tld \eta(\theta)$ is isomorphic to the restriction of $\epsilon^1_\theta\kappa_\theta$ to~$\mu_K \ltimes J^1_\theta$, since $\epsilon^1_\theta\kappa_\theta$ is the $p$-primary $\beta$-extension of~$\theta$. 
It follows from~\cite[Proposition~2.13(2, 3)]{InertialJL} that for a certain function~$\epsilon_\theta: \mu_K \to \{\pm 1\}$ one has
\[
\tr \kappa_\theta(x) = \epsilon^1_\theta(x) \epsilon_\theta(x) \dim \eta_{\theta_L}.
\]
It follows from this and~\cite[Proposition~2.14]{InertialJL} that if $x \in \mu_K$ and~$L = F[x]$ then
\begin{equation}\label{tracekappa}
\tr \kappa_\theta(x) = \epsilon^0_{\mu_K}(V_{\theta_L})\epsilon^0_{\mu_K}(V_\theta) \dim \eta_{\theta_L}
\end{equation}
where the~$\epsilon^0$ are signs, and the $\epsilon^1$ are quadratic characters of~$\mu_K$.
(We remark that by definition $\epsilon^1_\theta = \epsilon^1_{\mu_K}(-, V_{\theta})$ in the notation of~\cite{InertialJL}: compare the statement of~\cite[Theorem~4.10]{InertialJL}.)

Now consider the pair $(J_{\theta_L}, \lambda_L = \kappa_{\theta_L} \otimes (-1)^{n/[L[\beta]:F]+1}R_w(\chi))$ where the Lusztig induction is taken from~$\be_{n/\delta(\Theta_F)}^\times$ to the centralizer of the image of~$x$ in~$\GL_{n/[F[\beta]:F]}(\be)$, which is the group~$\GL_{n/[L[\beta]:F]}(\be[x])$. 
When~$\chi$ is~$\be[x]$-regular, this is a maximal simple type in~$G_L$. 
The corresponding $\bK_L$-type $\sigma^+_L = \Ind_{J_{\theta_L}}^{\bK_L}\lambda_L$ has dimension equal to
\begin{equation}\label{dimensionsigmaplus}
\dim(\sigma^+_L) = \dim \eta_{\theta_L}|J_{\theta_L} \backslash \bK_L | (\GL_{n/[L[\beta]:F]}(\be[x]):\be_{n/\delta(\Theta_F)}^\times)_{p'}.
\end{equation}

\begin{rk}\label{independencedimension}
The dimension of a virtual representation is the value of its character at the identity. 
In our case, it is independent of~$\chi$.
\end{rk}

\begin{pp}\label{traceGL(n)}
Let~$x \in \mu_K$ and let~$L = F[x]$. Then
\begin{displaymath}
\tr \sigma^+(x) = (-1)^{n/[F[\beta]:F] + n/[L[\beta]:F]} \epsilon^0_{\mu_K}(V_{\theta_L})\epsilon^0_{\mu_K}(V_\theta) \dim(\sigma^+_L) \sum_{\gamma \in \Gal(L/F)} \chi(\gamma x).
\end{displaymath}
\end{pp}

\begin{proof}
We begin with the Frobenius formula
\begin{displaymath}
\tr \sigma^+(x) = \sum_{J_\theta \backslash \bK} \tr \lambda(yxy^{-1}) =  \sum_{J_\theta \backslash \bK} (-1)^{n/[F[\beta]:F]+1}  R_w(\chi)(yxy^{-1})\tr \kappa_\theta(yxy^{-1})
\end{displaymath}
and the remark that if~$y \in \bK$ and $\tr \lambda(yxy^{-1}) \not = 0$ then there exists an element of~$J_\theta$ conjugating $yxy^{-1}$ to an element of~$\mu_{K}$. Indeed, by Lemma~\ref{pregularconjugacy} the pro-$p$-regular classes of~$J_\theta$ are in bijection with those of $\GL_{n/[F[\beta]:F]}(\be)$, and by Proposition~\ref{ellipticvanishing} the only ones on which~$ R_w(\chi)$ is nonzero are those represented in~$\mu_K$. It follows that~$J_\theta y = J_\theta \widetilde{y}$ for some~$\widetilde{y} \in N_\bK(L)$ (this is the normalizer of~$L$ for the conjugation action of~$\bK$). We have an isomorphism $N_\bK(L) / \bK_L \to \Gal(L/F)$, and the intersection $N_\bK(L) \cap J_\theta$ maps onto $\Gal(L / L \cap F[\beta])$. (To see this, notice that if~$x \in J_\theta = \fB_{\max}^\times J^1_\theta$ normalizes~$\mu_L$ then its image~$\overline{x}$ in~$\GL_{n/[F[\beta]:F]}(\be)$ normalizes the image of~$\mu_L$, and the automorphism of~$\mu_L$ induced by~$\ad\, \overline{x}$ determines that induced by~$\ad\, x$, hence is also induced by an element of~$\fB_{\max}^\times \cong \GL_{n/[F[\beta]:F]}(\mO_{F[\beta]})$.) 

It follows that the space $J_\theta N_\bK(L)$ decomposes into double cosets
\begin{displaymath}
J_\theta N_\bK(L) = \bigcup_{\sigma \in \Gal(L/F)} J_\theta t_\sigma \bK_L
\end{displaymath}
where $t_{\sigma} \in \bK$ induces~$\sigma$ on~$L$ by conjugation, and $J_\theta t_\sigma \bK_L = J_\theta t_\tau \bK_L$ if and only if $\tau\sigma^{-1} \in \Gal(L/L\cap F[\beta])$. Since $J_\theta t_\sigma \bK_L = J_{\theta} \bK_L t_\sigma$, we deduce that
\begin{displaymath}
\tr \sigma^+(x) = [L: F[\beta]\cap L]^{-1} \sum_{\gamma \in \Gal(L/F)} |J_{\theta_L} \backslash \bK_L| \tr \lambda(\gamma x).
\end{displaymath}
Recalling Formula~(\ref{tracekappa}), this is equal to
\begin{displaymath}
\tr \sigma^+(x) = [L:F[\beta]\cap L]^{-1} \epsilon^0_{\mu_K}(V_{\theta_L})\epsilon^0_{\mu_K}(V_\theta) \dim \eta_{\theta_L} \sum_{\gamma \in \Gal(L/F)} |J_{\theta_L} \backslash \bK_L| (-1)^{n/[F[\beta]:F]+1}R_w(\chi)(\gamma x).
\end{displaymath}
By Proposition~\ref{DLinduction}, and the fact that~$[F[\beta]:F] = \delta(\Theta_F)$, we have
\begin{displaymath}
(-1)^{n/[F[\beta]:F]+1}R_w(\chi)(x) = (-1)^{n/[F[\beta]:F] + n/[L[\beta]:F]}(\GL_{n/[L[\beta]:F]}(\be[x])):\be_{n/\delta(\Theta_F)}^\times)_{p'}\sum_{\alpha \in \Gal(\be[x]/\be)} \chi(\alpha x).
\end{displaymath}
Recall that $\be$ is isomorphic to the residue field of~$F[\beta]$, and since $L[\beta]/F[\beta]$ is an unramified extension generated by~$x$, $\be[x]$ is isomorphic to the residue field of~$L[\beta]$. The restriction map is an isomorphism 
\begin{displaymath}
\res: \Gal(L[\beta]/F[\beta]) \to \Gal(L / F[\beta] \cap L)
\end{displaymath} 
and it follows that
\begin{displaymath}
[L: F[\beta] \cap L]^{-1} \sum_{\gamma \in \Gal(L/F)} \sum_{\alpha \in \Gal(\be[x]/ \be )}\chi(\alpha\gamma x) = \sum_{\gamma \in \Gal(L/F)} \chi(\gamma x). 
\end{displaymath}
The claim now follows by Formula~(\ref{dimensionsigmaplus}).
\end{proof}

We end this section by proving an analogous result for~$D^\times$. 
Let $(J_\theta, \lambda = \kappa_\theta \otimes \chi)$ be a maximal simple type in~$D^\times$ and let $\sigma^+_D = \Ind_{J_\theta}^{\mO_D^\times}\lambda$ be the associated $\bK$-type. 
Fix a simple stratum~$[\mO_D, \beta]$ defining~$\theta$ and fix a maximal unramified extension~$K^+$ of~$F[\beta]$ in~$D' = Z_D(F[\beta])$. 
Let~$x \in \mu_{K^+}$ generate an extension $L/F$. 
Let~$\theta_L$ be the interior lift of~$\theta$ to~$L$, and write~$\lambda_L$ for any maximal simple type in~$Z_{D^\times}(L)$ with maximal simple character~$\theta_L$. 
Write~$\sigma^+_L = \Ind_{J_{\theta_L}}^{Z_{\mO_D^\times}(L)}\lambda_L$ for the corresponding $\bK_L$-type.

\begin{pp}\label{traceD}
We have an equality
\begin{displaymath}
\tr \sigma_D^+(x) = \epsilon_{\mu_K}^0(V_\theta)\epsilon_{\mu_K}^0(V_{\theta_L}) \dim(\sigma^+_L)\chi(x).
\end{displaymath}
\end{pp}
\begin{proof}
If~$y \in \mO_D^\times$ and $yxy^{-1} \in J_\theta = \mO_{D'}^\times J^1_\theta$, then $yxy^{-1}$ is $J_\theta$-conjugate to an element of~$\mu_{K^+}$, because $\mu_{K^+}$ represents the pro-$p$-regular conjugacy classes in~$J_\theta$ by Lemma~\ref{pregularconjugacy}. Again by Lemma~\ref{pregularconjugacy}, elements of~$\mu_{K^+}$ are pairwise nonconjugate in~$\mO_D^\times$. So $J_\theta y = J_\theta \widetilde{y}$ for some $\widetilde{y} \in Z_{\mO_D^\times}(L)$, and we deduce that
\begin{displaymath}
\tr \sigma^+_D(x) = |J_{\theta_L} \backslash Z_{\mO_D^\times}(L)|\tr \lambda(x).
\end{displaymath}
By the same argument as for~$\GL_n(F)$, we know that 
\begin{displaymath}
\epsilon^1_\theta(x)\tr \kappa_\theta(x) = \epsilon_\theta(x)\dim \eta_{\theta_L} = \epsilon_\theta(x) \dim(\lambda_L),
\end{displaymath} 
hence the claim follows from~\cite[Proposition~2.14]{InertialJL}.
\end{proof}

\paragraph{The formal degree formula.} Let~$\fs$ be a supercuspidal inertial class for~$\GL_n(F)$, and let~$\fs_D = \JL^{-1}(\fs)$ be its Jacquet--Langlands transfer to~$D^\times$. 
We give a relation between the dimension of a~$\bK$-type $\sigma^+_D$ for~$\fs_D$ and the dimension of a $\bK$-type~$\sigma^+$ for~$\fs$. 
We assume that~$\sigma_D$ and~$\sigma$ have been constructed as in the above. 
Write~$q = |\mbf|$ and $t(\fs) = t(\fs_D)$ for the torsion numbers of the inertial classes. Normalize the formal degrees for~$\GL_n(F)$ so that the Steinberg representation has formal degree one, and let~$\Iw \subseteq \bK$ be an Iwahori subgroup.

\begin{thm}\cite[(1.4.1)]{BHparametric}\label{formaldegreeformula}
The formal degree of any irreducible representation containing a maximal simple type~$(J_\theta, \lambda)$ corresponding to~$\fs$ is
\begin{displaymath}
d(\pi) = t(\fs) \dim(\lambda) \frac{q^n-1}{(q-1)^n} \frac{\mu_G(\Iw)}{\mu_G(J_\theta)}
\end{displaymath}
for any Haar measure~$\mu_G$ on~$G$.
\end{thm}

\begin{pp}\label{typedimension}
We have an equality $\dim(\sigma^+) = (\GL_n(\mbf):\mbf_n^\times)_{p'} \dim(\sigma^+_D)$.
\end{pp}
\begin{proof}
Multiplying numerator and denominator of the equation in Theorem~\ref{formaldegreeformula} by~$\mu_G(\bK)^{-1}$ yields
\begin{displaymath}
d(\pi) = t(\fs)\dim(\sigma^+)(\GL_n(\mbf):\mbf_n^\times)_{p'}^{-1}
\end{displaymath} 
because $\dim(\sigma^+) = \dim(\lambda)(\bK : J_\theta)$ and $(\bK : \Iw) = \frac{(q^n-1) \cdots (q^n-(q^{n-1}))}{q^{n(n-1)/2}(q-1)^n}$. We have seen that any irreducible representation in~$\fs_D$ restricts to~$\mO_D^\times$ to a sum of~$t(\fs_D)$ representations each appearing with multiplicity one and all conjugate under~$D^\times$, which are precisely the $\bK$-types for~$\fs_D$. Since $d(\pi) = \dim(\JL^{-1}\pi)$, we deduce that
\begin{displaymath}
t(\fs_D)\dim(\sigma^+_D) = t(\fs)\dim(\sigma^+)(\GL_n(\mbf):\mbf_n^\times)_{p'}^{-1}
\end{displaymath}
and the claim follows since~$t(\fs_D) = t(\fs)$, as the Jacquet--Langlands correspondence commutes with character twists.
\end{proof}

\section{Galois deformation theory.}\label{Galoisdeformationtheory} Working in the framework of \cite[Section~4]{EGBM}, we recall the definition of potentially semistable deformation rings of fixed Hodge type and discrete series Galois type, and prove some properties of the monodromy stratification. We state a form of the geometric Breuil--M\'ezard conjecture for the mod~$p$ fibers of these rings, and deduce a description of the cycle corresponding to discrete series lifts. In this section, we fix $p$-adic coefficients consisting of a finite extension $E/\bQ_p$ with ring of integers~$\mO_E$, uniformizer~$\pi_E$, and residue field~$\be$. We let $\rhobar: G_F \to \GL_n(\be)$ be a continuous representation, and we assume that~$E$ is sufficiently large (so that, for instance, it contains all $[F:\bQ_p]$ embeddings of~$F$).

\paragraph{Weights and algebraic representations.} Write~$\bZ^n_+$ for the set of $n$-tuples~$(\lambda_1, \ldots, \lambda_n)$ of integers such that $\lambda_1 \geq \ldots \geq \lambda_n$. This defines a dominant character $\diag(t_1, \ldots, t_n) \mapsto \prod_{i=1}^n t_i^{\lambda_i}$ of the diagonal torus in~$\GL_{n, F}$. There is an associated algebraic $\mO_F$-representation $M_{\lambda}' = \Ind_{B_n}^{\GL_n}(w_{\max} \lambda)_{/\mO_F}$ of~$\GL_{n, \mO_F}$ with highest weight~$\lambda$, for the upper-triangular Borel subgroup~$B_n$ and the longest element~$w_{\max}$ of the Weyl group. We write~$M_{\lambda}$ for the $\mO_F$-points of this representation. Then fix~$\lambda \in (\bZ_+^n)^{\Hom_{\bQ_p}(F, E)}$ and define an $\mO_E$-representation of~$\GL_n(\mO_F)$ by
\begin{displaymath}
L_{\lambda} = \bigotimes_{\tau: F \to E} \left ( M_{\lambda_\tau} \otimes_{\mO_F, \tau} \mO_E \right ).
\end{displaymath}

Next we recall some mod~$p$ representations. Given~$a \in \bZ_+^n$ with $p-1 \geq a_i - a_{i+1}$ for all $1 \leq i \leq n-1$, define
\begin{displaymath}
P_a = \Ind_{B_n}^{\GL_n}(w_{\max} a)_{/\mbf}(\mbf)
\end{displaymath}
and let~$N_a$ be the irreducible subrepresentation of~$P_a$ generated by a highest weight vector. The \emph{Serre weights} of~$\GL_n(\mbf)$ are the elements $a \in (\bZ_+^n)^{\Hom(\mbf , \be)}$ such that for all~$\sigma: \mbf \to \be$ we have $p-1 \geq a_{\sigma, i} - a_{\sigma, i+1}$ for~$1 \leq i \leq n-1$, and $0 \leq a_{\sigma, n} \leq p-1$. We furthermore require that not all $a_{\sigma, n} = p-1$. To a Serre weight there corresponds an irreducible $\be$-representation of~$\GL_n(\mbf)$, defined by
\begin{displaymath}
F_a = \bigotimes_{\tau \in \Hom(\mbf, \be)}\left ( N_{a_\tau} \otimes_{\mbf, \tau} \be \right ).
\end{displaymath}
These are absolutely irreducible and pairwise non-isomorphic, and every irreducible $\be$-representation of~$\GL_n(\mbf)$ has this form.

Finally, we introduce analogues for~$D^\times$. For every $\bQ_p$-linear embedding $\tau: F \to E$, fix an embedding $\tau^+: F_n \to E$ lifting~$\tau$ and write~$M_\lambda^+$ for the $\mO_{F_n}$-points of~$M_{\lambda}'$ (so that $M_{\lambda}^+ |_{\GL_n(\mO_F)}$ is isomorphic to $M_\lambda \otimes_{\mO_F} \mO_{F_n}$). Then we introduce
\begin{displaymath}
L_\lambda^+ = \bigotimes_{\tau: F \to E} \left ( M_{\lambda_\tau}^+ \otimes_{\mO_{F_n}, \tau^+} \mO_E \right )
\end{displaymath}
which has an action of $\mO_D^\times$ via a choice of $F_n$-linear isomorphism $j: D \otimes_F F_n \to M_n(F_n)$ mapping the order $\mO_D \otimes_{\mO_F} \mO_{F_n}$ into~$M_n(\mO_{F_n})$, and the inclusion $D \to D \otimes_F F_n, \; d \mapsto d \otimes 1$. We have the following lemma.

\begin{lemma}\label{weightcharacter}
If~$z_D \in \mO_D^\times$ corresponds to~$z \in \GL_n(\mO_F)$, in the sense that it is a semisimple element with the same characteristic polynomial as~$z$, then $\tr_{L_\lambda^+}(z_D) = \tr_{L_\lambda}(z)$.
\end{lemma}
\begin{proof}
This is because~$L_\lambda^+$ is a lattice in a $\GL_n(F_n)$-representation over~$E$, $\tr_{L_\lambda}(z) = \tr_{L_\lambda^+}(z)$, and~$z$ and~$z_D$ are conjugate in~$\GL_n(F_n)$ under any choice of~$j$.
\end{proof}

\paragraph{Inertial types and monodromy.}
An \emph{inertial type} for~$F$ is a smooth finite-dimensional representation of~$I_F$ that extends to a representation of~$W_F$. The type is a \emph{supercuspidal type} if it extends to an irreducible representation of~$W_F$, and a \emph{discrete series type} if it is a multiple of a supercuspidal inertial type. Two $n$-dimensional Weil--Deligne representations have the same restriction to inertia if and only if they are the Langlands parameters of irreducible representations of~$\GL_n(F)$ in the same inertial class. It follows that if~$\tau = \tau_0^{\oplus r}$ for a supercuspidal inertial type~$\tau_0$, there are a corresponding simple inertial class~$\fs$ for~$\GL_n(F)$ and representations~$\sigma_\fP(\tau) = \sigma_{\fP}(\fs)$ of~$\bK$ indexed by partitions~$\fP$ of~$r$. 
There are also virtual representations~$\sigma_{\fP}^+(\tau) = \sigma_{\fP}^+(\fs)$. 
Similarly, we define representations of~$\mO_D^\times$ by letting~$\sigma_D(\tau)$ be an arbitrary choice of $\bK$-type for~$\JL^{-1}(\fs)$ (we will prove our results for all possible choices). For~$\lambda \in (\bZ_+^n)^{\Hom_{\bQ_p}(F, E)}$ we put $\sigma_{\fP}(\tau, \lambda) = \sigma_{\fP}(\tau) \otimes L_\lambda$,  $\sigma_{\fP}^+(\tau, \lambda) = \sigma_{\fP}^+(\tau) \otimes L_\lambda$, and $\sigma_D(\tau, \lambda) = \sigma_D(\tau) \otimes L_\lambda ^+$.

Let~$\tau$ be a discrete series inertial type.
Our results will relate~$\sigma_{\fP_{\min}}^+(\tau)$ to the locus in the deformation space of~$\rhobar$ consisting of discrete series lift of inertial type~$\tau$, that is to say Galois representation lifting~$\rhobar$ whose associated Weil--Deligne representation is the Langlands parameter of an essentially square integrable representation in~$\fs$. Making this precise requires an account of the monodromy operator on the universal deformation ring. 

To start with, we recall some commutative algebra. Let~$A$ be a commutative ring with~$1$ and let~$M$ be a finite projective $A$-module of rank~$n$ with a nilpotent endomorphism~$N: A \to A$. To each prime ideal $x \in \Spec(A)$ we attach a partition~$\fP_x$ of~$n$ by considering the Jordan canonical form of the nilpotent endomorphism~$N(x)$ on~$M \otimes_A k(x)$, where~$k(x)$ is the residue field at~$x$. 

\begin{lemma}\label{monodromy}
Each partition~$\fP$ of~$n$ defines a closed subset of~$\Spec(A)$
\begin{displaymath}
\Spec(A)_{\geq \fP} = \{x \in \Spec(A) : \fP_x \geq \fP\}.
\end{displaymath}
\end{lemma}
\begin{proof}
See for instance~\cite[Section~4]{Pyvovarovthesis}. By our definition of~$\fP_x \geq \fP$ as the reverse of the dominance partial order on partitions, we find that $\fP_x \geq \fP$ if and only if~$\dim ( \ker N(x)^i) \geq \dim ( \ker N(\fP)^i)$ for all~$i$, where $N(\fP)$ has Jordan canonical form given by~$\fP$. Since $\dim ( \ker N(x)^i) = \dim (\coker N(x)^i)$ and $\coker N(x)^i \cong (\coker N^i) \otimes_A k(x)$, the claim follows since the set
\begin{displaymath}
\{x \in \Spec(A): \dim_{k(x)}((\coker N^i) \otimes_A k(x)) \geq m \}.
\end{displaymath}
is closed for all~$m \in \bZ$.
\end{proof}

\begin{rk}
It follows that if~$\Spec(A)$ is irreducible then the function~$x \mapsto \fP_x$ is constant on a dense open subset of~$\Spec(A)$, where it attains its minimal value. So we can define subsets~$\Spec(A)_{\fP}$ as the union of irreducible components of~$\Spec(A)$ where the minimal value of~$\fP_x$ is~$\fP$---equivalently, where the monodromy is generically~$\fP$. 
\end{rk}

\paragraph{Potentially semistable deformation rings.}
Let~$\tau: I_F \to \GL_n(E)$ be a discrete series inertial type and~$\lambda \in (\bZ^n_+)^{\Hom_{\bQ_p}(F, E)}$. Let $L/F$ be a finite Galois extension such that~$\tau$ is trivial on~$I_L$. By~\cite[Theorem~2.7.6]{Kisinsemistable} there is a quotient~$(R_{\rhobar}^\square[1/p])(\tau, \lambda)$ of the generic fibre of the universal lifting $\mO_E$-algebra~$R_{\rhobar}^\square$ whose points in a finite extension $E'/E$ correspond to potentially semistable lifts of~$\rhobar$ with Hodge type~$\lambda$ and inertial type~$\tau$. By~\cite[Theorem~2.5.5]{Kisinsemistable}, there is a finite projective $L_0 \otimes_{\bQ_p} (R_{\rhobar}^\square[1/p])(\tau, \lambda)$-module $D_{\rhobar}(\tau, \lambda)[1/p]$ with an automorphism~$\varphi$, semilinear with respect to~$\sigma \otimes 1$, and a $L_0 \otimes_{\bQ_p} (R_{\rhobar}^\square[1/p])(\tau, \lambda)$-linear nilpotent endomorphism~$N$, specializing to $D^*_{\mathrm{st}}(r^{\univ}_x|_{G_L})$ for any $\mO_E$-linear ring homomorphism $x: (R_{\rhobar}^\square[1/p])(\tau, \lambda) \to E'$. Since $D_{\rhobar}(\tau, \lambda)[1/p]$ is a direct factor of a free $L_0 \otimes_{\bQ_p} (R_{\rhobar}^\square[1/p])(\tau, \lambda)$-module, it is also projective over $(R_{\rhobar}^\square[1/p])(\tau, \lambda)$. By Lemma~\ref{monodromy} we have a stratification~$\Spec(R_{\rhobar}^\square[1/p])(\tau, \lambda)_{\geq \fP}$, and $\Spec(R_{\rhobar}^\square[1/p])(\tau, \lambda)_{\geq \fP_{\max}}$ corresponds to the vanishing of the monodromy operator, hence to potentially crystalline deformations of~$\rhobar$ (recall that~$\fP_{\max}$ is the partition $n = 1 + \cdots + 1$).

We write $(R_{\rhobar}^\square[1/p])(\tau, \lambda)_{\fP}$ for the reduced quotient corresponding to the set $\Spec(R_{\rhobar}^\square[1/p])(\tau, \lambda)_{\fP}$, and we let $R_{\rhobar}(\tau, \lambda)_{\fP}$ be the image of~$R^\square_{\rhobar} \to (R_{\rhobar}^\square[1/p])(\tau, \lambda)_{\fP}$. This is a reduced $\pi_E$-torsion free $\mO_E$-algebra whose generic fibre is isomorphic to $(R_{\rhobar}^\square[1/p])(\tau, \lambda)_{\fP}$, and its minimal primes have characteristic zero (by $\mO_E$-flatness), hence they are in bijection with those of the generic fibre, which are the components where the monodromy is generically~$\fP$. 
(By definition~$R_\rhobar(\tau,\lambda)_\fP$ is the Zariski closure in~$R_\rhobar^\square$ of the set of these components of the generic fibre.)
We define $R_{\rhobar}(\tau, \lambda)_{\geq \fP}$ similarly. By~\cite[Theorem~3.3.4]{Kisinsemistable}, these rings are equidimensional of the same dimension~$d$.

\paragraph{Cycles.} Since the rings $R_{\rhobar}(\tau, \lambda)_{\fP}$ are equidimensional and $\pi_E$-torsion free, their special fibres~$R_{\rhobar}(\tau, \lambda)_{\fP}/\pi_E$ are also equidimensional, and define a $(d-1)$-cycle on~$R^\square_{\rhobar}$ by~\cite[Lemma~2.1]{BMraffinees}. The geometric conjecture in~\cite[Section~4.2]{EGBM} states that for each Serre weight~$a$ for~$\GL_n(\mbf)$ there exists a cycle~$\mC_a$ on~$R^\square_{\rhobar}$ such that
\begin{displaymath}
Z(R_{\rhobar}(\tau, \lambda)_{\fP_{\max}}/\pi_E) = \sum_{a} n_a \mC_a
\end{displaymath}
where the multiplicity~$n_a$ is equal to the multiplicity of the representation~$F_a$ in~$\sigmabar_{\fP_{\max}}(\tau, \lambda)$, the semisimplified mod~$\pi_E$ reduction of~$\sigma_{\fP_{\max}}(\tau, \lambda)$. Notice that $R_{\rhobar}(\tau, \lambda)_{\fP_{\max}}$ is a potentially crystalline deformation ring of~$\rhobar$. This can be reformulated by defining a group homomorphism
\begin{displaymath}
\overline{\cyc}: R_{\be}(\GL_n(\mbf)) \to Z^{d-1}(R^\square_{\rhobar}), \; F_a \mapsto \mC_a
\end{displaymath}
and one can generalize the statement of the conjecture, and ask whether
\begin{displaymath}
Z(R_{\rhobar}(\tau, \lambda)_{\geq \fP}/\pi_E) = \overline{\cyc}(\sigmabar_{\fP}(\tau, \lambda)).
\end{displaymath}
This is motivated by the fact that $\sigma(\tau)_\fP$ is contained in a generic irreducible representation~$\pi$ of~$\GL_n(F)$ if and only if the inertial class of~$\pi$ corresponds to~$\tau$ and the partition~$\fP(\pi)$ attached to~$\pi$ satisfies~$\fP(\pi) \geq \fP$, that is
\begin{displaymath}
\Hom_{\bK}(\sigma_{\fP}(\tau), \pi) \not = 0 \; \text{if and only if} \; \rec(\pi)|_{I_K} \cong \tau \; \text{and} \; \fP(\pi) \geq \fP.
\end{displaymath}
Under some assumptions on~$\rhobar$, this is true when $F = \bQ_p$ and $n = 2$ by~\cite{KisinFM} or when $n = 2$ and~$\lambda = 0$ by~\cite{GKBM}. However, we expect that this statement has to be modified for~$n \geq 3$ to account for multiplicities: it is not true in general that $\Hom_{\bK}(\sigma_{\fP}(\tau), \pi)$ is one-dimensional when it is nonzero. The general statement should be
\begin{displaymath}
Z(R_{\rhobar}(\tau, \lambda)_{\fP}/\pi_E) = \overline{\cyc}(\sigmabar^+_{\fP}(\tau, \lambda)),
\end{displaymath} 
because of the multiplicities
\begin{align*}
\dim \Hom_{\bK}(\sigma_{\fP}^+(\tau), \pi)  = &\; 1 \; \text{if} \; \rec(\pi)|_{I_K} \cong \tau \; \text{and} \; \fP(\pi) = \fP\\
& \; 0 \; \text{otherwise.}
\end{align*}
We offer two pieces of evidence towards this. 
The first is our main result, concerning the case~$\fP_{\min}$, which gives a compatibility with the analogous statement on central division algebras. 
Second, observe that $\dim_E \Hom_{\bK}(\sigma_{\fP}(\tau), \pi_{\fP'}(\tau))$ equals the Kostka number~$K_{\fP, \fP'}$, and so we have an equality in the Grothendieck group
\begin{displaymath}
\sigmabar_{\fP}(\tau, \lambda) = \sum_{\deg \fP' = \deg \fP} K_{\fP, \fP'} \sigmabar_{\fP'}^+(\tau, \lambda)
\end{displaymath}
and
\begin{displaymath}
\sigmabar^+_{\fP}(\tau, \lambda) = \sum_{\deg \fP' = \deg \fP} K^+_{\fP, \fP'}\sigmabar_{\fP'}(\tau, \lambda)
\end{displaymath}
where~$(K_{\fP, \fP'}^+)$ is the inverse of the matrix $(K_{\fP, \fP'})$ of Kostka numbers. Now~\cite[Corollary~4.9]{ShottonBM} says that the direct analogues of our formulas give the right answer for deformation rings with $\ell$-adic coefficients, where $\ell \not = p$ is a prime number. This is also consistent with the work of Yao described in the introduction.

\section{Jacquet--Langlands transfers.}\label{JLtransfer}

In this section we construct a Jacquet--Langlands transfer of Serre weights from~$D^\times$ to~$\GL_n(F)$, and prove its compatibility with the inertial Jacquet--Langlands correspondence.
We also consider analogues for $\ell$-adic coefficients when~$\ell \ne p$, so we begin by fixing a prime number~$\ell$ (allowing, of course, the case $\ell = p$). 
We will mostly be interested in proving that our transfer preserves congruences of types.
However, we point out that in the case of trivial regular weight we can interpret our result as describing a Jacquet--Langlands correspondence for representations of maximal compact subgroups of~$D^\times$ and~$\GL_n(F)$.
This is because of the following lemma.

\begin{lemma}\label{typesrestriction}
Let~$R$ be an algebraically closed field of any characteristic (including~$\operatorname{char} R = p$) and let~$\tau$ be an irreducible smooth $R$-linear representation of~$\mO_D^\times$. Then~$\tau$ occurs in the restriction to~$\mO_D^\times$ of an irreducible smooth representation of~$D^\times$.
\end{lemma}
\begin{proof}
We regard~$\tau$ as a representation of~$F^\times \mO_D^\times$ with~$\pi_F$ acting trivially. As in~\cite[Section~4]{Vignerasl1}, $\tau$~extends to a representation~$\tau'$ of its normalizer $N = N_{D^\times}(\tau)$, and the induction $\Ind_{N}^{D^\times}(\tau')$ is an irreducible representation of~$D^\times$ containing~$\tau$.
\end{proof}

Choosing an isomorphism $\iota_\ell: \cbQ_\ell \to \bC$, one gets a Jacquet--Langlands transfer from inertial classes of $\cbQ_\ell$-representations of~$D^\times$ to inertial classes of~$\cbQ_\ell$-representations of~$\GL_n(F)$. 
Because the Harish--Chandra character is compatible with automorphisms of the coefficient field, this transfer is independent of the choice of $\iota_\ell$~\cite[10.1]{MScongruences}. 

\begin{defn}\label{JLtransfermap}
Define a map $\JL_{\bK}: R_{\cbQ_\ell}(\mO_D^\times) \to R_{\cbQ_\ell}(\GL_n(\mO_F))$ as follows. 
Let~$\sigma_D$ be an irreducible representation of~$\mO_D^\times$. 
Then~$\sigma_D$ is a type for some Bernstein component~$\fs_D$ of~$D^\times$, by Lemma~\ref{typesrestriction}, and we let~$\fs = \JL(\fs_D)$. 
We define $\JL_{\bK}(\sigma_D) = \sigma_{\fP_{\min}}^+(\fs)$. 
\end{defn}

\paragraph{Mod~$p$ reduction.} Set $\ell = p$. We construct a map
\begin{displaymath}
\JL_p : R_{\cbF_p}(\mO_D^\times) \to R_{\cbF_p}(\GL_n(\mO_F))
\end{displaymath}
and prove our main result, namely that $\JL_p(\sigmabar_D(\tau, \lambda)) = \sigmabar_{\fP_{\min}}^+(\tau, \lambda)$. Since every irreducible smooth $\cbF_p$-representation of a pro-$p$ group is trivial, it is enough to define a map
\begin{displaymath}
\JL_p: R_{\cbF_p}(\bd^\times) \to R_{\cbF_p}(\GL_n(\mbf)).
\end{displaymath}
We choose any $\mbf$-linear isomorphism $\iota: \bd \to \mbf_n$ and we define~$\JL_p$ to be the semisimplified mod~$p$ reduction of~$\chi \mapsto (-1)^{n+1}R_w(\chi)$, composed with the isomorphism $R_{\cbF_p}(\mbf_n^\times) \to R_{\cbQ_p}(\mbf_n^\times)$. Since~$R_w$ is constant on $\Gal(\mbf_n/\mbf)$-orbits, this is independent of the choice of~$\iota$.  
Recall the explicit formula~\ref{DLinduction}, and observe that~$\JL_p$ is a direct generalization of the construction in Section~2 of~\cite{GGBM}.

For any profinite group~$G$, one defines the Brauer character of a finite-dimensional representation~$V$ of~$G$ over a finite field~$\bF_q$ as in the finite group case, obtaining a function~$\chi(V)$ on the set of pro-$p$-regular conjugacy classes of~$G$ valued in~$\cbQ_p$. 
From Lemma~\ref{pregularconjugacy}, and the corresponding assertion for finite groups, we find that whenever~$G$ has an open normal pro-$p$ subgroup the Brauer character induces an isomorphism $R_{\cbF_p}(G) \otimes_{\bZ} \cbQ_p \to \mC^{(p)}(G, \cbQ_p)$, where~$R_{\cbF_p}(G)$ is the Grothendieck group of finite length smooth representations of~$G$ over~$\cbF_p$, and the target denotes the space of functions from the set of pro-$p$-regular classes of~$G$ to~$\cbQ_p$. We get an induced map
\begin{displaymath}
\JL_p: \mC^{(p)}(\mbf_n^\times, \cbQ_p) \to \mC^{(p)}(\GL_n(\mbf), \cbQ_p)
\end{displaymath}
such that if~$x \in \GL_n(\mbf)$ has a conjugate in~$\mbf_n^\times$ with degree~$\deg(x)$ over~$\mbf$ then
\begin{equation}\label{JLfunctions}
\JL_p(f)(x) = (-1)^{n+n/\deg(x)}(\GL_{n/\deg(x)}(\mbf_{\deg(x)}) : \mbf_n^\times)_{p'} \sum_{\gamma \in \Gal(\mbf_{\deg(x)}/\mbf)} f(\gamma x)
\end{equation}
by Proposition~\ref{DLinduction}.

\begin{thm}\label{comparison}
Let~$\tau$ be a discrete series inertial type for~$I_F$ and~$\lambda \in (\bZ_+^n)^{\Hom_{\bQ_p}(F, E)}$. Then we have the equality
\begin{displaymath}
\JL_p(\sigmabar_D(\tau, \lambda)) = \sigmabar_{\fP_{\min}}^+(\tau, \lambda)
\end{displaymath}
\end{thm}
\begin{proof}
We have an equality of Brauer characters 
\[
\chi(\sigmabar_{\fP_{\min}}^+(\tau, \lambda)) = \chi(\overline{L}_{\lambda})\chi(\sigmabar^+_{\fP_{\min}}(\tau)), 
\]
and similarly 
\[
\chi(\sigmabar_D(\tau, \lambda)) = \chi(\overline{L}_{\lambda}^+)\chi(\sigmabar_D(\tau)). 
\]
The representation $\sigma^+_{\fP_{\min}}(\tau)$ is smooth and defined over a finite extension $E/ \bQ_p$, so we can compute~$\chi(\sigmabar^+_{\fP_{\min}}(\tau))$ as the restriction of the trace of~$\sigma^+_{\fP_{\min}}(\tau)$ to $p$-regular conjugacy classes: this follows from the corresponding statement in the finite group case, via Lemma~\ref{pregularconjugacy}. By Proposition~\ref{ellipticvanishing}, both $\chi(\sigmabar_D(\tau))$ and $\chi(\sigmabar^+_{\fP_{\min}}(\tau))$ vanish away from certain conjugacy classes represented by roots of unity. If~$z$ and~$z_D$ are matching $p$-regular roots of unity, then Lemma~\ref{weightcharacter} actually implies that $\chi(\overline{L}_{\lambda})(z) = \chi(\overline{L}_{\lambda}^+)(z_D)$, because the Brauer character of a representation of the finite groups generated by~$z$ and~$z_D$ can be computed on a lift to characteristic zero. Hence it is enough to prove that $\JL_p \left ( \chi(\sigmabar_D(\tau)) \right ) = \chi(\sigmabar^+_{\fP_{\min}}(\tau))$. 

Let~$\fs(\tau)$ be the simple inertial class corresponding to~$\tau$.
Let~$\Theta_F = \cl(\fs(\tau))$, and recall that we have fixed a lift~$\Theta_E \to \Theta_F$ to the unramified parameter field.
If~$\theta$ is a maximal simple character in~$\GL_n(F)$ with endo-class~$\Theta_F$, this gives rise to a conjugacy class of isomorphisms
\[
J_\theta/J^1_\theta \isom \GL_{n/\delta(\Theta_F)}(\be).
\]
We fix a simple stratum~$[\fA, \beta]$ for~$\theta$, and a maximal unramified extension $K^+/F[\beta]$ in~$Z_A(F[\beta])$ such that $J_\theta \subseteq \bK$ and the maximal unramified extension~$K$ of~$F$ in~$K^+$ normalizes the group~$\bK$. 
Let~$[\chi] = \Lambda(\fs(\tau), \Theta_E, \kappa_\theta)$, so that by Proposition~\ref{thetaindependence} we have
\[
\sigma^+_{\fP_{\min}}(\tau) \cong \Ind_{J_\theta}^{\bK}(\kappa_\theta \otimes (-1)^{n/\delta(\Theta_F)+1}R_w(\chi)).
\]
By the main results of~\cite{InertialJL}, the invariants of~$\fs_D(\tau) = \JL^{-1}(\fs(\tau))$ are
\[
\cl(\fs_D(\tau)) = \Theta_F \text{ and } \Lambda(\fs_D(\tau), \Theta_E, \kappa_{\theta_D}) = [\chi].
\]
It follows that we can choose a maximal simple character~$\theta_D$ in~$D^\times$ with $\cl(\theta) = \cl(\theta_D)$ and a simple stratum~$[\mO_D, \beta_D]$ for~$\theta_D$ such that~$\sigma_D(\tau)$ is isomorphic to the induction $\Ind_{J_{\theta_D}}^{\mO_D^\times}(\kappa_{\theta_D} \otimes \chi)$.
We fix a maximal unramified extension~$K^+_D/F[\beta_D]$ in~$Z_D(F[\beta_D])$ and write~$K_D$ for the maximal unramified extension of~$F$ in~$K^+_D$. Since the Jacquet--Langlands correspondence preserves torsion numbers, we have $[K:F] = [K_D:F]$, and there exists a unique isomorphism $\iota: K_D \to K$ such that the equality of endo-classes $\cl(\theta_{D, K}) = \iota^* \cl(\theta_K)$ holds.

Let~$z \in \mu_ K$ and~$z_D \in \mu_{K_D}$ generate isomorphic extensions of~$F$, which we identify via~$\iota$ with an unramified extension~$L/F$. By Propositions~\ref{traceGL(n)} and~\ref{traceD} we have equalities
\begin{equation}\label{tracesplit}
\tr \sigma_{\fP_{\min}}^+(\tau)(z) = (-1)^{n/[F[\beta]:F] + n/[L[\beta]:F]}\epsilon_{\mu_K}^0(V_{\theta_L})\epsilon_{\mu_K}^0(V_\theta)\dim(\sigma_L^+)\sum_{\gamma \in \Gal(L/F)}\chi(\gamma z)
\end{equation}
and 
\begin{equation}\label{tracenonsplit}
\tr \sigma_D(\tau)(z_D) = \epsilon_{\mu_K}^0(V_{\theta_{D, L}})\epsilon_{\mu_K}^0(V_{\theta_D})\dim(\sigma_{D, L}^+)\chi(x).
\end{equation}
These compute the Brauer characters of the mod~$p$ reductions $\sigmabar_{\fP_{\min}}^+(\tau)$ and~$\sigmabar_D(\tau)$ at~$z$ and~$z_D$. It follows that
\begin{displaymath}
\JL_p \left ( \sigmabar_D(\tau)(z) \right ) = (-1)^{n+n/[L:F]}(\GL_{n/[L:F]}(\mbf_{[L:F]}):\mbf_n^\times)_{p'}\epsilon^0_{\mu_K}(V_{\theta_{D, L}})\epsilon^0_{\mu_K}(V_{\theta_D})\dim(\sigma_{D, L}^+)\sum_{\gamma \in \Gal(L/F)}\chi(\gamma z_D)
\end{displaymath}
by Formula~(\ref{JLfunctions}), and we have to compare this to~(\ref{tracesplit}).

Recall from Remark~\ref{independencedimension} that~$\dim(\sigma_L^+)$ and~$\dim(\sigma_{D, L}^+)$ are equal to the dimensions of the $\bK$-types corresponding to an arbitrary choice of \emph{maximal} simple types with maximal simple character~$\theta_L$, respectively~$\theta_{D, L}$. 
By our choice of~$\iota: F[z_D] \to F[z]$, these characters have the same endo-class, hence we can choose maximal simple types with simple characters~$\theta_L$ and~$\theta_{D, L}$ that determine inertial classes corresponding to each other under the Jacquet--Langlands correspondence between $Z_{D^\times}(F[z_D])$ and $Z_{\GL_n(F)}(F[z])$ (identified with groups over~$L$ via~$\iota$). 
By Proposition~\ref{typedimension}, this implies that
\begin{displaymath}
\dim(\sigma_{D, L}^+)(\GL_{n/[L:F]}(\mbf_{[L:F]}):\mbf_n^\times)_{p'} = \dim(\sigma_L^+)
\end{displaymath}
since~$\mbf_{[L:F]}$ is isomorphic to the residue field of~$L$ and $\mbf_n^\times \cong ((\mbf_{[L:F]})_{n/[L:F]})^\times$.
Finally, the computations at the end of the proof of~\cite[Theorem~4.10]{InertialJL} show that\footnote{The integers there denoted~$m$ are all equal to one, since~$D$ is a division algebra.}
\begin{displaymath}
(-1)^{n+n/[K:F]+n/[F[\beta]:F]}\epsilon_{\mu_K}^0(V_\theta) = -\epsilon_{\mu_K}^0(V_{\theta_D})
\end{displaymath}
and
\begin{displaymath}
(-1)^{n/[L:F] + n/[K:F] + n/[L[\beta]:F]}\epsilon_{\mu_K}^0(V_{\theta_L}) = -\epsilon_{\mu_K}^0(V_{\theta_{D, L}}).
\end{displaymath}
\end{proof}

We remark that when the weight $\lambda = 0$, Theorem~\ref{comparison} implies that the following diagram
\begin{equation}
\begin{tikzcd}
R_{\cbQ_p}(\mO_D^\times) \arrow[r, "\JL_{\bK}"] \arrow[d, "\br_p"] & R_{\cbQ_p}(\GL_n(\mO_F)) \arrow[d, "\br_p"]\\
R_{\cbF_p}(\mO_D^\times) \arrow[r, "\JL_p"] & R_{\cbF_p}(\GL_n(\mO_F))
\end{tikzcd}
\end{equation}
commutes.

\paragraph{Mod~$\ell$ reduction.} 
Now assume $\ell \not = p$. 
By our discussion of $\bK$-types for~$D^\times$, Lemma~\ref{typesrestriction} implies that every irreducible smooth $\cbQ_\ell$-representation~$\tau$ of~$\mO_D^\times$ is a $\bK$-type for a Bernstein component of~$D^\times$. As such, there exists a simple character~$\theta$ such that $\tau \cong \Ind_{J_\theta}^{\mO_D^\times}(\kappa_\theta \otimes \chi)$ for some character~$\chi$ of~$J_\theta/J^1_\theta$, and the $\mO_D^\times$-conjugacy class of the maximal simple type $(J_\theta, \kappa_\theta \otimes \chi)$ is uniquely determined by~$\tau$.

\begin{lemma}\label{reductionreps}
Every irreducible $\cbF_\ell$-representation of~$D^\times$ is the mod~$\ell$ reduction of a $\cbQ_\ell$-representation of~$\mO_D^\times$. The mod~$\ell$ reduction of an irreducible $\cbQ_\ell$-representation~$\tau$ of~$\mO_D^\times$ is irreducible.
\end{lemma}
\begin{proof}
Since~$\mO_D^\times$ is a solvable group, the first claim is a consequence of the Fong--Swan theorem~\cite[Theorem~38]{Serrereps}. For the second claim, observe first that $\tau|_{1+\fp_D}$ is a direct sum with multiplicity one of representations forming a unique~$\mO_D^\times$-orbit. Indeed, by~\cite[Proposition~4.1]{Vignerasl1} the group $\Hom_{1+\fp_D}(\sigma, \tau)$ is a simple module for the Hecke algebra $\mH(\mO_D^\times, \sigma)$, for every representation~$\sigma$ of~$1+\fp_D$. Since the quotient $\mO_D^\times/1+\fp_D$ is cyclic, by~\cite[Proposition~4.2]{Vignerasl1} this Hecke algebra is commutative, hence its simple $\cbQ_\ell$-modules are one-dimensional, proving the claim of multiplicity one.
Now, if~$\tau^0$ is any $\cbZ_\ell$-lattice in~$\tau$ then the reduction $\overline{\tau}^0$ will again be a direct sum with multiplicity one of irreducible $\cbF_\ell$-representations of~$1+\fp_D$, because $1+\fp_D$ is a pro-$p$ group, and~$\mO_D^\times$ will act transitively on the summands. 
Hence every irreducible $\mO_D^\times$-subrepresentation of~$\overline{\tau}^0$ has to coincide with~$\overline{\tau}^0$.
\end{proof}

\begin{thm}~\label{modlJL} There exists a unique map~$\JL_\ell$ making the following diagram
\begin{equation}
\begin{tikzcd}
R_{\cbQ_\ell}(\mO_D^\times) \arrow[r, "\JL_{\bK}"] \arrow[d, "\br_\ell"] & R_{\cbQ_\ell}(\GL_n(\mO_F)) \arrow[d, "\br_\ell"]\\
R_{\cbF_\ell}(\mO_D^\times) \arrow[r, "\JL_\ell"] & R_{\cbF_\ell}(\GL_n(\mO_F))
\end{tikzcd}
\end{equation}
 commute.
\end{thm}
\begin{proof}
The mod~$\ell$ reduction map for $\cbQ_\ell$-representations is defined as the direct limit of the reduction maps over finite extensions of~$\bQ_\ell$. That~$\JL_\ell$ is unique follows from the first claim in Lemma~\ref{reductionreps}, since the left vertical arrow is surjective. For the existence, by Lemma~\ref{reductionreps} it suffices to prove that if $\tau_1$ and~$\tau_2$ are irreducible representations of~$\mO_D^\times$ with the same mod~$\ell$ reduction, then $\br_\ell(\JL_{\bK}(\tau_1)) = \br_\ell(\JL_{\bK}(\tau_2))$. Indeed, this allows us to define $\JL_{\ell}(\sigmabar)$ as~$\br_{\ell} \, \JL_{\bK}(\sigma)$ for any irreducible lift~$\sigma$ of~$\sigmabar$, and then commutativity of the diagram holds by definition and the second part of Lemma~\ref{reductionreps}.

Since $\br_\ell(\tau_1) = \br_\ell(\tau_2)$, we have $\tau_1 \cong \tau_2 \otimes \psi$ for some character $\psi: \mO_D^\times/1+\fp_D \to \cbQ_\ell^\times$, because the restrictions $\tau_i|_{1+\fp_D}$ are isomorphic modulo~$\ell$, hence they are isomorphic over~$\cbQ_\ell$ as $1+\fp_D$ is a pro-$p$ group. Hence there exists a simple character~$\theta_D$ with endo-class~$\Theta_F$ such that $\tau_i = \Ind_{J_{\theta_D}}^{\mO_D^\times}(\kappa_{\theta_D} \otimes \chi_i)$ (where the~$\chi_i$ are computed with respect to a lift $\Theta_E \to \Theta_F$). By assumption, the representations $\br_\ell(\kappa_{\theta_D} \otimes \chi_i)$ intertwine in~$\mO_D^\times$, as they have isomorphic inductions to~$\mO_D^\times$. Since~$\kappa_{\theta_D}$ is a $\beta$-extension, the intertwining set of~$\kappa_{\theta_D}$ in~$D^\times$ coincides with that of~$\theta_D$, which is also equal to its normalizer $\pi_{D'}^{\bZ} \ltimes J_{\theta_D}$ (where we have fixed a parameter field $F[\beta]$ for~$\theta_D$, and $D' = Z_D(F[\beta])$). Hence we see that $\br_\ell[\chi_1] = \br_\ell[\chi_2]$, where~$[\chi_i]$ denotes the orbit under $\Gal(\be_{n/\delta(\Theta_F)}/\be)$.

There exists a maximal simple character~$\theta$ in~$\GL_n(F)$ with the same endo-class as~$\theta_D$, together with a conjugacy class of isomorphisms $J_\theta/J^1_\theta \to \GL_{n/\delta(\Theta_F)}(\be)$ induced by $\Theta_E \to \Theta_F$. We assume that the subgroup~$J_\theta$ is contained in~$\bK$, so that the virtual representation $\JL_{\bK}(\tau_i)$ is the induction $\Ind_{J_\theta}^{\bK}(\kappa_\theta \otimes (-1)^{n/\delta(\Theta_F) + 1}R_w(\chi_i))$.

To conclude, it suffices to prove that $\br_\ell R_w(\chi_1) = \br_\ell R_w(\chi_2)$, or that the $\ell$-Brauer characters of the~$R_w(\chi_i)$ coincide. These are the restrictions to $\ell$-regular classes in~$\GL_{n/\delta(\Theta_F)}(\be)$ of the characters of the~$R_w(\chi_i)$. An element of~$\GL_{n/\delta(\Theta_F)}(\be)$ is $\ell$-regular if and only if its semisimple part is~$\ell$-regular, because the unipotent elements of this group have order a power of~$p$. The character formula of Deligne and Lusztig (see~\cite[Theorem~4.2]{DLreps}) expresses the value of~$R_w(\chi_i)$ at~$g \in \GL_{n/\delta(\Theta_F)}(\be)$ with Jordan decomposition~$g = su$ in terms of a Green function evaluated at~$u$ (this is independent of~$\chi_i$) and the value of~$\chi_i$ at those conjugates of~$s$ contained in the inducing torus. 
If~$s$ is an $\ell$-regular element these character values for~$\chi_1$ and~$\chi_2$ coincide since we have seen that $[\chi_1^{(\ell)}] = [\chi_2^{(\ell)}]$ (because their mod~$\ell$ reductions are the same).
\end{proof}

\section{Breuil--M\'ezard conjectures.}\label{BMconjectures}

Fix a prime number~$\ell$, possibly equal to~$p$, and a finite extension $E/ \bQ_\ell$. Let~$\rhobar: \Gal(\overline{F}/F) \to \GL_n(\be)$ be a continuous representation. Let~$R^\square_{\rhobar}$ be the framed deformation ring of~$\rhobar$ over~$\mO_E$.

\paragraph{Case~$\ell = p$.} 
We prove the following more precise form of the first theorem in the introduction.
\begin{thm} 
Assume that~$E$ is large enough that all irreducible $\cbF_p$-representations of~$\GL_n(\mO_F)$ and~$\mO_D^\times$ are defined over~$\be$. Assume the geometric Breuil--M\'ezard conjecture for~$\GL_n(F)$, i.e. the existence of a homomorphism
\begin{displaymath}
\overline{\cyc}: R_{\be}(\GL_n(\mO_F)) \to Z^{d-1}(R^\square_{\rhobar}/\pi_E)
\end{displaymath}
such that $\overline{\cyc}(\sigmabar_{\fP_{\min}}^+(\tau, \lambda)) = Z(R^\square_{\rhobar}(\tau, \lambda)_{\fP_{\min}}/\pi_E)$ whenever~$\tau$ is defined over~$E$. 
Then there exists a homomorphism
\[
\lbar{\cyc}_{D^\times}: R_{\be}(\mO_D^\times) \to Z^{d-1}(R^\square_{\rhobar}/\pi_E)
\]
such that $\overline{\cyc}_{D^\times}(\sigmabar_D(\tau, \lambda)) = Z(R^\square_{\rhobar}(\tau, \lambda)_{\fP_{\min}}/\pi_E)$ whenever~$\tau$ is defined over~$E$.
\end{thm}
\begin{proof}
By Theorem~\ref{comparison}, it suffices to define
\begin{displaymath}
\overline{\cyc}_{D^\times} = \overline{\cyc} \circ \JL_p.
\end{displaymath}
\end{proof}
\begin{rk}
The statement of the Breuil--M\'ezard conjecture for~$\GL_n(F)$ that we assume in the previous theorem is given in~\cite[Conjecture~4.2.1]{EGBM} in the crystalline case and \cite[Conjecture~1.5.1]{LLHLMmodels} in the semistable case.
Given the Breuil--M\'ezard conjecture for~$\GL_n(F)$, our result gives a description of the mod~$p$ fibres of discrete series lifting rings in terms of the representation theory of~$\mO_D^\times$ and the type theory of~$D^\times$.
\end{rk}

\paragraph{Case $\ell \not = p$.}
In this case, we may not find a finite extension $E/\bQ_l$ such that all irreducible $\be$-representations of~$\mO_D^\times$ are absolutely irreducible. 
We assume that~$E$ is large enough that whenever~$\rhobar$ has a lift of inertial type~$\tau$ to some finite extension of~$\bQ_l$, then~$\tau$ and all the corresponding $\bK$-types for~$\GL_n(F)$ and~$D^\times$ are defined over~$E$. 
We also assume that~$E$ and $k_E$ are large enough that all irreducible components of~$\Spec(R^\square_{\rhobar}[1/p])$ and $\Spec(R^\square_{\rhobar}/\pi_E)$ are geometrically irreducible. 
For any pair~$(\tau, N)$ consisting of an inertial type and a monodromy operator, write $R^\square_{\rhobar}(\tau, N)$ for the corresponding quotient of the $\mO_E$-deformation ring $R^\square_{\rhobar}$, as in~\cite{ShottonBM}. 
The characteristic zero points of~$R^\square_{\rhobar}(\tau, N)$ correspond generically to lifts of~$\rhobar$ whose attached Weil--Deligne representation has inertial type~$\tau, N$. 
Define a map 
\begin{displaymath}
\mathrm{cyc} : R_E(\GL_n(\mO_F)) \to Z^d(R^\square_{\rhobar}), \;\; \sigma \mapsto \sum_{\tau, N} \dim_{\cbQ_\ell} \Hom_{\cbQ_\ell[\GL_n(\mO_F)]}(\sigma^\vee \otimes_E \cbQ_\ell, \pi_{\tau, N}) [R^\square_{\rhobar}(\tau, N)]
\end{displaymath}
where~$\pi_{\tau, N}$ is any irreducible generic $\cbQ_\ell$-representation of~$\GL_n(F)$ such that $\rec_{\cbQ_\ell}(\pi_{\tau, N})$ has inertial type~$\tau, N$.
The map~$\rec_{\cbQ_\ell}$ is only well-defined up to the choice of a square root of~$q$ in~$\cbQ_\ell$, but this plays no role when considering the inertial type. 
Similarly, we introduce a map
\begin{displaymath}
\mathrm{cyc}_{D^\times} : R_E(\mO_D^\times) \to Z^d(R^\square_{\rhobar}), \;\; \sigma \mapsto \sum_{\tau, N} \dim \Hom_{\cbQ_\ell[\GL_n(\mO_F)]}(\sigma^\vee \otimes_E \cbQ_\ell, \JL^{-1}(\pi_{\tau, N})) [R^\square_{\rhobar}(\tau, N)].
\end{displaymath}
In this formula we set~$\JL^{-1}(\pi) = 0$ when~$\pi$ is a generic representation that is not essentially square-integrable (this is consistent with the fact that the Langlands--Jacquet transfer is nonzero on elliptic representations only, and the only generic elliptic representations are the essentially square-integrable representations. See~\cite{Datelliptiques}.)

\begin{thm}[Breuil--M\'ezard conjecture for~$D^\times$, case~$\ell \not = p$]\label{modlBM}
Assume $p \not = 2$. There exists a unique map~$\overline{\cyc}_{D^\times, \ell}$ making the following diagram commute.
\begin{equation}\label{BMdivisionalgebra}
\begin{tikzcd}
R_E(\mO_D^\times) \arrow[d, "\br_{\ell}"] \arrow[r, "\cyc_{D^\times}"] & Z^d(R^\square_{\rhobar}) \arrow[d, "\mathrm{red}"] \\
R_{k_E}(\mO_D^\times) \arrow[r, "\overline{\cyc}_{D^\times, \ell}"] & Z^{d-1}(R^\square_{\rhobar}/ \pi_E) 
\end{tikzcd}
\end{equation}
\end{thm}
\begin{proof}
Since the map~$\br_\ell$ is surjective for~$\mO_D^\times$, it suffices to prove that if~$x \in \ker(\br_\ell)$ then $x \in \ker(\red \circ \cyc_{D^\times})$. 
This says that every congruence between $\bK$-types gives rise to a congruence between deformation rings: it is not a formal statement. 

By~\cite[Theorem~4.6]{ShottonBM}, there exists a commutative diagram
\begin{equation}\label{BMGL(n)}
\begin{tikzcd}
R_E(\GL_n(\mO_F)) \arrow[d, "\br_{\ell}"] \arrow[r, "\cyc"] & Z^d(R^\square_{\rhobar}) \arrow[d, "\mathrm{red}"] \\
R_{k_E}(\GL_n(\mO_F)) \arrow[r, "\overline{\cyc}_\ell"] & Z^{d-1}(R^\square_{\rhobar}/ \pi_E) .
\end{tikzcd}
\end{equation}
Let~$x_{\cbQ_\ell}$ be the image of~$x$ in~$R_{\cbQ_\ell}(\mO_D^\times)$. 
Fix a finite extension $L/E$ large enough that all irreducible summands of~$x_{\cbQ_\ell}$ and $\JL_{\bK}(x_{\cbQ_\ell})$ are defined over~$L$. 
Then $\cyc_{D^\times}(x_{\cbQ_\ell}^\vee) = \cyc (\JL_{\bK}(x_{\cbQ_\ell})^\vee)$, where we regard $\JL_{\bK}(x_{\cbQ_\ell})$ as an element of~$R_L(\GL_n(\mO_F))$ and the two sides as cycles on the deformation ring with $\mO_L$-coefficients. 
Indeed, if~$\sigma$ is an $L$-representation of~$\mO_D^\times$ then we have by construction the equality
\begin{displaymath}
\dim \Hom_{\cbQ_\ell[\mO_D^\times]}(\sigma_{\cbQ_\ell}, \JL^{-1}(\pi_{\tau, N})) = \dim \Hom_{\cbQ_\ell[\GL_n(\mO_F)]}(\JL_{\bK}(\sigma_{\cbQ_\ell}), \pi_{\tau, N})
\end{displaymath}
because this equality holds on $\bK$-types for~$\mO_D^\times$, and by Lemma~\ref{typesrestriction} the $\bK$-types span $R_{\cbQ_\ell}(\mO_D^\times)$. 
Because of our assumptions on~$E$, the natural maps $Z^d(R^\square_{\rhobar}) \to Z^d(R^\square_{\rhobar} \otimes_{\mO_E} \mO_L)$ and $Z^{d-1}(R^\square_{\rhobar}/\pi_E) \to Z^{d-1}(R^\square_{\rhobar} \otimes_{\mO_E} k_L)$ are isomorphisms, hence it suffices to prove that 
\[
\red \, \cyc \, \JL_{\bK}(x_{\cbQ_\ell})^\vee = 0. 
\]
Since diagram~(\ref{BMGL(n)}) commutes  (working with $L$-coefficients in the diagram), we have that $\red \, \cyc \, \JL_{\bK}(x_{\cbQ_\ell})^\vee = \overline{\cyc}_{\ell} \,  \br_\ell \, \JL_{\bK}(x_{\cbQ_\ell})^\vee$. 
By Theorem~\ref{modlJL}, we have $\br_\ell \, \JL_{\bK}(x_{\cbQ_\ell})^\vee = (\br_\ell \, \JL_{\bK}(x_{\cbQ_\ell}))^\vee = (\JL_\ell \, \br_\ell (x_{\cbQ_\ell}))^\vee = 0$, and the claim follows.
\end{proof}

\bibliographystyle{amsalpha}
\bibliography{refpapers}

\providecommand{\bysame}{\leavevmode\hbox to3em{\hrulefill}\thinspace}
\providecommand{\MR}{\relax\ifhmode\unskip\space\fi MR }
\providecommand{\MRhref}[2]{%
  \href{http://www.ams.org/mathscinet-getitem?mr=#1}{#2}
}
\providecommand{\href}[2]{#2}
\begin{thebibliography}{LLHLM23}

\bibitem[BH96]{BHliftingI}
Colin~J. Bushnell and Guy Henniart, \emph{Local tame lifting for {${\rm GL}(N)$}. {I}. {S}imple characters}, Inst. Hautes \'Etudes Sci. Publ. Math. (1996), no.~83, 105--233. \MR{1423022}

\bibitem[BH04]{BHparametric}
\bysame, \emph{Local {J}acquet--{L}anglands correspondence and parametric degrees}, Manuscripta Math. \textbf{114} (2004), no.~1, 1--7. \MR{2136524}

\bibitem[BH10]{BHLLIII}
\bysame, \emph{The essentially tame local {L}anglands correspondence, {III}: the general case}, Proc. Lond. Math. Soc. (3) \textbf{101} (2010), no.~2, 497--553. \MR{2679700}

\bibitem[BH11]{BHJL}
\bysame, \emph{The essentially tame {J}acquet--{L}anglands correspondence for inner forms of {${\rm GL}(n)$}}, Pure Appl. Math. Q. \textbf{7} (2011), no.~3, Special Issue: In honor of Jacques Tits, 469--538. \MR{2848585}

\bibitem[BK93]{BKbook}
Colin~J. Bushnell and Philip~C. Kutzko, \emph{The admissible dual of {${\rm GL}(N)$} via compact open subgroups}, Annals of Mathematics Studies, vol. 129, Princeton University Press, Princeton, NJ, 1993. \MR{1204652}

\bibitem[BM02]{BMmultiplicites}
Christophe Breuil and Ariane M{\'e}zard, \emph{Multiplicit\'es modulaires et repr\'esentations de {${\rm GL}_2({\bf Z}_p)$} et de {${\rm Gal}(\overline{\bf Q}_p/{\bf Q}_p)$} en {$l=p$}}, Duke Math. J. \textbf{115} (2002), no.~2, 205--310, With an appendix by Guy Henniart. \MR{1944572}

\bibitem[BM14]{BMraffinees}
\bysame, \emph{Multiplici{t\'e}s modulaires raffi{n\'e}es}, Bull. Soc. Math. France \textbf{142} (2014), no.~1, 127--175. \MR{3248725}

\bibitem[BSS12]{BSSV}
P.~Broussous, V.~S{\'e}cherre, and S.~Stevens, \emph{Smooth representations of {${\rm GL}_m(D)$} {V}: {E}ndo-classes}, Doc. Math. \textbf{17} (2012), 23--77. \MR{2889743}

\bibitem[Car85]{Carterbook}
Roger~W. Carter, \emph{Finite groups of {L}ie type}, Pure and Applied Mathematics (New York), John Wiley \& Sons, Inc., New York, 1985, Conjugacy classes and complex characters, A Wiley-Interscience Publication. \MR{794307}

\bibitem[CK]{CKJL}
Przemyslaw Chojecki and Erick Knight, \emph{p-adic {J}acquet--{L}anglands correspondence and patching}, \url{arXiv:1709.10306}.

\bibitem[Dat07]{Datelliptiques}
Jean-Fran\c{c}ois Dat, \emph{Th{\'e}orie de {L}ubin-{T}ate non-ab{\'e}lienne et repr{\'e}sentations elliptiques}, Invent. Math. \textbf{169} (2007), no.~1, 75--152. \MR{2308851}

\bibitem[DL76]{DLreps}
P.~Deligne and G.~Lusztig, \emph{Representations of reductive groups over finite fields}, Ann. of Math. (2) \textbf{103} (1976), no.~1, 103--161. \MR{0393266}

\bibitem[DM83]{DMLusztigGreen}
Fran{\c c}ois Digne and Jean Michel, \emph{Foncteur de {L}usztig et fonctions de {G}reen g\'en\'eralis\'ees}, C. R. Acad. Sci. Paris S\'er. I Math. \textbf{297} (1983), no.~2, 89--92. \MR{720915}

\bibitem[DM91]{DMbook}
Fran{\c{c}}ois Digne and Jean Michel, \emph{Representations of finite groups of {L}ie type}, London Mathematical Society Student Texts, vol.~21, Cambridge University Press, Cambridge, 1991. \MR{1118841}

\bibitem[Dot22]{InertialJL}
Andrea Dotto, \emph{The inertial {J}acquet-{L}anglands correspondence}, J. Reine Angew. Math. \textbf{784} (2022), 177--214. \MR{4388335}

\bibitem[EG14]{EGBM}
Matthew Emerton and Toby Gee, \emph{A geometric perspective on the {B}reuil-{M\'e}zard conjecture}, J. Inst. Math. Jussieu \textbf{13} (2014), no.~1, 183--223. \MR{3134019}

\bibitem[GG15]{GGBM}
Toby Gee and David Geraghty, \emph{The {Breuil--M\'{e}zard} conjecture for quaternion algebras}, Annales de l'Institut Fourier (2015), 1557--1575.

\bibitem[GK14]{GKBM}
Toby Gee and Mark Kisin, \emph{The {Breuil--{M}\'ezard} conjecture for potentially {B}arsotti-{T}ate representations}, Forum Math. Pi \textbf{2} (2014), e1, 56. \MR{3292675}

\bibitem[Kis08]{Kisinsemistable}
Mark Kisin, \emph{Potentially semi-stable deformation rings}, J. Amer. Math. Soc. \textbf{21} (2008), no.~2, 513--546. \MR{2373358}

\bibitem[Kis09]{KisinFM}
\bysame, \emph{The {F}ontaine--{M}azur conjecture for {${\rm GL}_2$}}, J. Amer. Math. Soc. \textbf{22} (2009), no.~3, 641--690. \MR{2505297}

\bibitem[Kis10]{Kisinstructure}
\bysame, \emph{The structure of potentially semi-stable deformation rings}, Proceedings of the {I}nternational {C}ongress of {M}athematicians. {V}olume {II}, Hindustan Book Agency, New Delhi, 2010, pp.~294--311. \MR{2827797}

\bibitem[LLHLM23]{LLHLMmodels}
Daniel Le, Bao~V. Le~Hung, Brandon Levin, and Stefano Morra, \emph{Local models for {G}alois deformation rings and applications}, Invent. Math. \textbf{231} (2023), no.~3, 1277--1488. \MR{4549091}

\bibitem[Lus77]{LusztigCoxeter}
G.~Lusztig, \emph{Coxeter orbits and eigenspaces of {F}robenius}, Invent. Math. \textbf{38} (1976/77), no.~2, 101--159. \MR{0453885}

\bibitem[MS14]{MStypes}
Alberto M{\'i}nguez and Vincent S{\'e}cherre, \emph{Types modulo {$\ell$} pour les formes int{\'e}rieures de {${\rm GL}_n$} sur un corps local non archim{\'e}dien}, Proc. Lond. Math. Soc. (3) \textbf{109} (2014), no.~4, 823--891, With an appendix by Vincent S{\'e}cherre et Shaun Stevens. \MR{3273486}

\bibitem[MS17]{MScongruences}
\bysame, \emph{Correspondance de {J}acquet--{L}anglands locale et congruences modulo {$\ell$}}, Invent. Math. \textbf{208} (2017), no.~2, 553--631. \MR{3639599}

\bibitem[Pa{\v s}05]{Paskunasunicity}
Vytautas Pa{\v s}k{\=u}nas, \emph{Unicity of types for supercuspidal representations of {${\rm GL}_N$}}, Proc. London Math. Soc. (3) \textbf{91} (2005), no.~3, 623--654. \MR{2180458}

\bibitem[Pyv21]{Pyvovarovthesis}
Alexandre Pyvovarov, \emph{On the {B}reuil-{S}chneider conjecture generic case}, Algebra Number Theory \textbf{15} (2021), no.~2, 309--339. \MR{4243650}

\bibitem[Roc09]{Rocheottawa}
Alan Roche, \emph{The {B}ernstein decomposition and the {B}ernstein centre}, Ottawa lectures on admissible representations of reductive {$p$}-adic groups, Fields Inst. Monogr., vol.~26, Amer. Math. Soc., Providence, RI, 2009, pp.~3--52. \MR{2508719}

\bibitem[Sch18]{ScholzeLT}
Peter Scholze, \emph{On the {$p$}-adic cohomology of the {L}ubin-{T}ate tower}, Ann. Sci. \'{E}c. Norm. Sup\'{e}r. (4) \textbf{51} (2018), no.~4, 811--863, With an appendix by Michael Rapoport. \MR{3861564}

\bibitem[Ser77]{Serrereps}
Jean-Pierre Serre, \emph{Linear representations of finite groups}, Springer-Verlag, New York-Heidelberg, 1977, Translated from the second French edition by Leonard L. Scott, Graduate Texts in Mathematics, Vol. 42. \MR{0450380}

\bibitem[Sho18]{ShottonBM}
Jack Shotton, \emph{The {B}reuil--{M}{\'e}zard conjecture when {$l\neq p$}}, Duke Math. J. \textbf{167} (2018), no.~4, 603--678. \MR{3769675}

\bibitem[SS19]{SecherreStevensJL}
V.~S{\'e}cherre and S.~Stevens, \emph{Towards an explicit local {J}acquet-{L}anglands correspondence beyond the cuspidal case}, Compos. Math. \textbf{155} (2019), no.~10, 1853--1887. \MR{4000000}

\bibitem[SZ99]{SZtypes}
P.~Schneider and E.-W. Zink, \emph{{$K$}-types for the tempered components of a {$p$}-adic general linear group}, J. Reine Angew. Math. \textbf{517} (1999), 161--208, With an appendix by Schneider and U. Stuhler. \MR{1728541}

\bibitem[SZ00]{SZcharacters}
Allan~J. Silberger and Ernst-Wilhelm Zink, \emph{The characters of the generalized {S}teinberg representations of finite general linear groups on the regular elliptic set}, Trans. Amer. Math. Soc. \textbf{352} (2000), no.~7, 3339--3356. \MR{1650042}

\bibitem[Vig01]{Vignerasl1}
Marie-France Vign{\'e}ras, \emph{La conjecture de {L}anglands locale pour {${\GL}(n,F)$} modulo {$\ell$} quand {$\ell \not= p, \ell >n$}}, Ann. Sci. \'Ecole Norm. Sup. (4) \textbf{34} (2001), no.~6, 789--816. \MR{1872421}

\end{thebibliography}

\end{document}